\documentclass{amsart}
\usepackage{color,amssymb}

\newtheorem{theorem}{Theorem}[section]

\newtheorem{lemma}[theorem]{Lemma}
\newtheorem{example}[theorem]{Example}
\newtheorem{proposition}[theorem]{Proposition}
\theoremstyle{definition}
\newtheorem{definition}[theorem]{Definition}
\theoremstyle{remark}
\newtheorem{remark}[theorem]{Remark}
\numberwithin{equation}{section}

\newcommand{\Real}{\mathbb R}
\newcommand{\Complex}{\mathbb C}

\newcommand{\To}{\longrightarrow}

\newcommand{\tu}{\tilde{u}}
\newcommand{\pil}{\pi_{\lambda}}
\newcommand{\pas}{\partial_s}

\newcommand{\ve}{\varepsilon}

\newcommand{\C}{\mathbb{C}}
\newcommand{\R}{\mathbb{R}}

\begin{document}

\title[Holomorphic Open Book Decompositions]{Holomorphic Open Book Decompositions}%
\author{Casim Abbas}%
\address{Department of Mathematics, Michigan State University, East Lansing, MI 48824}%
\email{abbas@math.msu.edu}%


\begin{abstract}
Emmanuel Giroux showed that every contact structure on a closed three dimensional manifold is supported by an open book decomposition. We will extend this result by showing that the open book decomposition can be chosen in such a way that the pages are solutions to a homological perturbed holomorphic curve equation. 
\end{abstract}
\maketitle
\tableofcontents
\section{Introduction}
This paper is the starting point of a larger program by the author, Helmut Hofer and Samuel Lisi investigating a perturbed holomorphic curve equation in the symplectisation of a three dimensional contact manifold \cite{AHL-1}, \cite{AHL-2}. One aim of this program is to provide an alternative proof of the Weinstein conjecture in dimension three as outlined in \cite{ACH} complementing Clifford Taubes' gauge theoretical proof \cite{Taubes-1}, \cite{Taubes-2}. A special case of this paper's main result has been used in the proof of the Weinstein conjecture for planar contact structures in \cite{ACH}. Another reason for studying this equation is to construct foliations by surfaces of section with nontrivial genus. This is usually impossible to do with the unperturbed holomorphic curve equation since solutions generically do not exist.\\
Consider a closed three dimensional manifold $M$ equipped with a contact form
$\lambda$. This is a 1-form which satisfies $\lambda\wedge d\lambda\neq 0$ at every point of $M$. We denote the associated contact structure by $\xi=\ker\lambda$ and
the Reeb vector field by $X_{\lambda}$. Recall that the Reeb vector field is defined by the two equations 
\[
i_{X_{\lambda}}d\lambda\,=\,0\ \ \mbox{and}\ \ i_{X_{\lambda}}\lambda\,=\,1.
\]
\begin{definition}{\bf (Open Book Decomposition)}\\
Assume that $K\subset M$ is a link in $M$, and that $\tau:M\backslash K\rightarrow S^1$ is a fibration so
that the fibers $F_{\vartheta}=\tau^{-1}({\vartheta})$ are interiors of compact embedded surfaces $\bar{F}_{\vartheta}$ with $\partial\bar{F}_{\vartheta}=K$, where $\vartheta$ is the coordinate along $K$.  We also assume that $K$ has a tubular neighborhood $K\times D$, $D\subset{\mathbb R}^2$ being the open unit disk, such that $\tau$ restricted to $K\times (D\backslash\{0\})$ is given by
$\tau(\vartheta,r,\phi)=\phi$, where $(r,\phi)$ are polar coordinates on $D$.\\
Then we call $\tau$ an {\sl open book decomposition} of $M$, the link $K$ is called the {\sl binding} of the open book decomposition, and the surfaces $F_{\vartheta}$ are called the {\sl pages} of the open book decomposition.
\end{definition}
It is a well-known result in three dimensional topology that every closed three dimensional orientable manifold admits an open book decomposition. Indeed, J. Alexander proved the following theorem in 1923, see \cite{Alexander} or \cite{Rolfsen}:
\begin{theorem}\label{Alexanderstheorem}
Every closed, orientable manifold $M$ of dimension $3$ is
diffeomorphic to $$W(h)\,{\bigcup}_{\mbox{Id}}\,(\partial W\times D^2)$$ where $D^2$ is the closed unit disk in ${\mathbb R}^2$, $W$ is an
orientable surface with boundary and $h:W\rightarrow W$ is an
orientation preserving diffeomorphism which restricts to the
identity near $\partial W$. $W(h)$ denotes the manifold obtained
from $W\times[0,2\pi]$ by identifying $(x,0)$ with $(h(x),2\pi)$.
\end{theorem}\qed\\

The above decomposition is an open book decomposition, the pages are given by
\[
 F_{\vartheta}:=(W\times\{\vartheta\})\,{\bigcup}_{\mbox{Id}}\,(\partial W\times I_{\vartheta})\ ,\ 0\le\vartheta<2\pi,
\]
where $I_{\vartheta}:=\{re^{i\vartheta}\in D^2\,|\,0<r< 1\}$, and the binding is given by $K=\partial W\times\{0\}$. Note that we allow $\partial W$ to be disconnected. \\

Emmanuel Giroux introduced the notion of an open book decomposition supporting a contact structure:
\begin{definition}{\bf (Supporting Open Book Decomposition \cite{Giroux})}\\
Assume that $M$ is a closed three dimensional manifold endowed with a contact form $\lambda$. Let $\tau$ be an open book decomposition with binding $K$. We say that $\tau$ {\sl supports the contact structure} $\xi$ if there exists a contact form $\lambda'$ with the same kernel as $\lambda$ so that $d\lambda'$  induces an area-form on each fiber
$F_{\vartheta}$ with $K$ consisting of closed orbits of the Reeb vector field $X_{\lambda'}$, and
$\lambda'$ orients $K$ as the boundary of
$(F_{\vartheta},d\lambda')$. 
\end{definition}

We will refer to a contact form $\lambda'$ above as a '{\sl Giroux contact form}'. Note that $\lambda'$ is not unique and that it is in general different from the original contact form $\lambda$. The following theorem by E. Giroux guarantees existence of such open book decompositions, and it contains a uniqueness statement as well, see \cite{Giroux}:

\begin{theorem}\label{giroux}
Every co-oriented contact structure $\xi=\ker\lambda$ on a closed three dimensional manifold is supported by some open book. Conversely, if two contact structures are supported by the same open book then they are diffeomorphic.
\end{theorem}\qed\\

In the topological category it is possible to modify an open book decomposition such that the pages of the new decomposition have lower genus at the expense of increasing the number of connected components of $K$. It was not known for some time whether a similar statement can also be made in the context of supporting open book decompositions. In particular, it was unclear whether every contact structure is supported by an open book decomposition whose pages are punctured spheres ({\sl planar pages}). The author and his collaborators could resolve the Weinstein conjecture for contact forms inducing a 'planar contact structure' in 2005 (see \cite{ACH}). So the question whether all contact structures are planar became a priority, which prompted John Etnyre to address it in \cite{Etnyre}. He showed that overtwisted contact structures always admit supporting open book decompositions with planar pages, but many contact structures do not. Since then planar open book decompositions have become an important tool in contact geometry.\\

In this paper we will prove that every contact structure has a supporting open book decomposition such that the pages solve a {\it homological perturbed Cauchy-Riemann type equation} which we will now describe after introducing some notation. We write $\pi_{\lambda}=\pi:TM\rightarrow \xi$ for the projection along the Reeb vector field $X_{\lambda}$. Fix a complex multiplication $J:\xi\rightarrow\xi$ so that the map
$\xi\oplus\xi\rightarrow {\mathbb R}$ defined by
$$
(h,k)\rightarrow d\lambda(h,Jk)
$$
defines a positive definite metric on the fibers. We will call such complex multiplications compatible (with $d\lambda$). The equation of interest here is the following
nonlinear first order elliptic system. The solutions consist of 5-tuplets $(S,j,\Gamma,\tilde{u},\gamma)$ where $(S,j)$ is a closed Riemann surface with complex structure $j$, $\Gamma\subset S$ is a finite subset, $\tilde{u}=(a,u):\dot{S}\rightarrow {\mathbb R}\times
M$ is a proper map with $\dot{S}=S\setminus\Gamma$, and $\gamma $ is a one-form on
$S$ so that
\begin{equation}\label{eq1}
\left\{
\begin{array}{ccc}
&\pi\circ Tu\circ j = J\circ \pi\circ Tu\ \hbox{ on }\ \dot{S}&\\
&(u^{\ast}\lambda)\circ j = da +\gamma\ \hbox{ on }\ \dot{S}&\\
& d\gamma= d(\gamma\circ j)=0\ \hbox{ on }\ S&\\
&E(\tilde{u})<\infty.&
\end{array}
\right.
\end{equation}
Here the {\it energy} $E(\tilde{u})$ is defined by
\begin{eqnarray*}
E(\tilde{u})=\sup_{\varphi\in\Sigma}\ \int_{\dot{S}}\
\tilde{u}^{\ast}d(\varphi\lambda),
\end{eqnarray*}
where $\Sigma$ consists of all smooth maps $\varphi:{\mathbb
R}\rightarrow [0,1]$ with $\varphi'(s)\geq 0$ for all $s\in
{\mathbb R}$.

Note that the above equation reduces to the usual pseudoholomorphic curve equation in the symplectisation $\R\times M$ if we set $\gamma=0$. The following proposition, which is a modification of a result
by Hofer, \cite{Hofer-Weinstein-conj}, shows that solutions to the problem (\ref{eq1}) approach cylinders over periodic orbits of the Reeb vector field.

\begin{proposition}
Let $(M,\lambda)$ be a closed three-dimensional manifold equipped
with a contact form $\lambda$. Then the associated Reeb vector
field has periodic orbits if and only if the associated PDE-problem (\ref{eq1}) has a non-constant solution.
\end{proposition}

\begin{proof}
Let $(S,j,\Gamma,\tilde{u},\gamma)$ be a non-constant solution of (\ref{eq1}).
If $\Gamma\neq \emptyset$ then the results in \cite{Hofer-Weinstein-conj} and \cite{HWZ-2} imply that near a puncture the solution is asymptotic to a periodic orbit (see also \cite{A-notes} for a complete proof). Here we use that $\gamma$ is exact near the punctures. The aim is now to show that in the absence of punctures the map $a$ is constant while the image of $u$ lies on a periodic Reeb orbit. Assume that $\Gamma=\emptyset$.\\
Since
$$
u^{\ast}\lambda = -da\circ j -\gamma\circ j,
$$
we find after applying $d$ that
$$
\Delta_j a =-d(da\circ j)=u^{\ast}d\lambda.
$$
In view of the equation $\pi\circ Tu\circ j=J\circ \pi\circ Tu$ we
see that $u^{\ast}d\lambda$ is a non-negative integrand. Applying
Stokes' theorem we obtain $\int_S u^{\ast}d\lambda =0$ implying
that
$$
\pi\circ Tu\equiv 0.
$$
Hence $a$ is a harmonic function on $S$ and therefore constant. So far, we also know that the image of $u$ lies on a Reeb trajectory, and it remains to show that this trajectory is actually periodic.\\
Let $\tau:\tilde{S}\rightarrow S$ be the universal covering map. The complex structure $j$ lifts to a complex structure $\tilde{j}$ on $\tilde{S}$. Pick now smooth functions $f,g$ on $\tilde{S}$ such that
\[
dg=\tau^{\ast}\gamma=:\tilde{\gamma}\ ,\ -df=\tau^{\ast}(\gamma\circ j)=\tilde{\gamma}\circ\tilde{j}.
\]
Then the map $u\circ \tau:\tilde{S}\rightarrow M$ satisfies
\[
(u\circ\tau)^{\ast}\lambda\,=\,df.
\]
The image of $u\circ \tau$ lies on a trajectory $x$ of the Reeb vector field in view of
\[
D(u\circ\tau)(z)\zeta\,=\,Df(z)\zeta\cdot X_{\lambda}((u\circ\tau)(z)),
\]
hence $(u\circ\tau)(z)=x(h(z))$ for some smooth function $h$ on $\tilde{S}$, and it follows that, after maybe adding a constant to $f$, we have
\[
(u\circ\tau)(z)\,=\,x(f(z)).
\]
The function $f$ does not descent to $S$. If it did it would have to be constant since it is harmonic. On the other hand this would imply that $u$ is constant in contradiction to our assumption that it is not. Therefore, there is a point $q\in S$ and two lifts $z_0,z_1\in\tilde{S}$ such that $f(z_0)>f(z_1)$. Let $\ell:S^1\rightarrow S$ be a loop which lifts to a path $\alpha:[0,1]\rightarrow\tilde{S}$ with $\alpha(0)=z_0$ and $\alpha(1)=z_1$. Considering the map
\[
v:=u\circ\ell:S^1\To M
\]
we see that $v(t)=(u\circ\tau\circ\alpha)(t)=x(f(\alpha(t)))$ and $x(f(z_0))=x(f(z_1))$, i.e. the trajectory $x$ is a periodic orbit. Hence the image of $u$ is a periodic orbit for the Reeb vector
field.
\end{proof}

The following is the main result of this paper.
\begin{theorem}\label{main-thm}
Let $M$ be a closed three dimensional manifold, and let $\lambda'$ be a contact form on $M$. Then the following holds for a suitable contact form $\lambda=f\,\lambda'$ where $f$ is a positive function on $M$: There exists a smooth family $(S,j_{\tau},\Gamma_{\tau},\tilde{u}_{\tau}=(a_{\tau},u_{\tau}),\gamma_{\tau})_{\tau\in S^1}$ of solutions to (\ref{eq1}) for a suitable compatible complex structure $J:\ker\lambda\rightarrow\ker\lambda$ such that 
\begin{itemize}
\item all maps $u_{\tau}$ have the same asymptotic limit $K$ at the punctures, where $K$ is a finite union of periodic trajectories of the Reeb vector field $X_{\lambda}$,
\item $$u_{\tau}(\dot{S})\cap u_{\tau'}(\dot{S})=\emptyset\ \ \mbox{if}\ \ \tau\neq\tau'$$
\item $$M\backslash K\,=\,\bigcup_{\tau\in S^1}u_{\tau}(\dot{S})$$
\item the projection $P$ onto $S^1$ defined by $p\in u_{\tau}(\dot{S})\mapsto \tau$ is a fibration
\item The open book decomposition given by $(P,K)$ supports the contact structure $\ker\lambda$, and $\lambda$ is a Giroux form.
\end{itemize}  
\end{theorem}

Here is a very brief outline of the argument. The reader is invited to skip foward to the section 'Conclusion' to see in more detail how all the partial results of this paper are tied together to prove the main result. In section \ref{existence} we will find a Giroux contact form which has a certain normal form near the binding. Following an argument by Chris Wendl \cite{Wendl}, \cite{Wendl-2} we will then {\it almost} be able to turn the Giroux leaves into solutions of (\ref{eq1}) without harmonic form except for the fact that we have to accept a confoliation form instead of a contact form. Pick one of these Giroux leaves as a starting point. The next step is to prove a result which permits us to perturb the Giroux leaf into a genuine solution of (\ref{eq1}) while simultaneously perturbing the confoliation form slightly into a contact form. This is where the harmonic form in (\ref{eq1}) is required. We actually obtain a local family of nearby solutions, not just one. In section \ref{compactness-section} we prove a compactness result which extends the local family of solutions into a global one. The remarkable fact is that there is a compactness result in the context of this paper although there is none in general for the perturbed holomorphic curve equation. The special circumstances in this paper imply a crucial apriori bound which implies that a sequence of solutions has a pointwise convergent subsequence with a measureable limit. The objective is then to show that the regularity of this limit is much better, it is actually smooth.\\

We consider two solutions $(S,j,\Gamma,\tilde{u},\gamma)$ and
$(S',j',\Gamma',\tilde{u}',\gamma')$ equivalent if there exists a
biholomophic map $\phi:(S,j)\rightarrow (S',j')$ mapping $\Gamma$
to $\Gamma'$ (preserving the enumeration) so that
$\tilde{u}'\circ\phi=\tilde{u}$. We will often identify a solution $(S,j,\Gamma,\tu,\gamma)$ of (\ref{eq1}) with its equivalence class
$[S,j,\Gamma,\tilde{u},\gamma]$. We note that we have a natural
${\mathbb R}$-action on the solution set by associating to
$c\in{\mathbb R}$ and $[S,j,\Gamma,\tilde{u},\gamma]$ the new
solution
$$
c+[S,j,\Gamma,\tilde{u},\gamma]= [S,j,\Gamma,(a+c,u),\gamma]\ ,\ \tu=(a,u).
$$
A crucial concept for our discussion will be the notion of a
finite energy foliation ${\mathcal F}$.
\begin{definition}{\bf (Finite Energy Foliation)}\\
A foliation ${\mathcal F}$ of ${\mathbb R}\times M$ is called a
{\it finite energy foliation} if every leaf $F$ is the image of an
embedded solution $[S,j,\Gamma,\tilde{u},\gamma]$ of the equations (\ref{eq1}), i.e.
$$
F=\tilde{u}(\dot{S}),
$$
so that $u(\dot{S})\subset M$ is transverse to the Reeb vector field, and for every leaf $F\in{\mathcal F}$ we also have $c+F\in{\mathcal F}$
for every $c\in{\mathbb R}$, i.e. the foliation is ${\mathbb
R}$-invariant.
\end{definition}

We recall the concept of a global surface of section. Let $M$ be a closed three-manifold
and $X$ a nowhere vanishing smooth vector field.

\begin{definition}{\bf (Surface of section)}\\
a) A local surface of section for $(M,X)$ consists of an embedded
compact surface $\Theta\subset M$ with boundary, so that
$\partial\Theta$ consists of a finite union of periodic orbits
(called the binding orbits). In addition the interior
$\dot{\Theta}=\Theta\setminus\partial\Theta$
is transverse to the flow. \\
b) A local surface of section is called a global surface of
section if in addition every orbit other than a binding orbit hits
$\dot{\Theta}$ in forward and backward time. In addition the
globally defined return map $\Psi:\dot{\Theta}\rightarrow \dot{\Theta}$ has a
bounded return time, i.e. there exists a constant $c>0$ so that
every $x\in\dot{\Theta}$ hits $\dot{\Theta}$ again in forward time not
exceeding $c$.
\end{definition}

Using proposition \ref{localnormalform-2} below, the existence part of Giroux's theorem can be rephrased as follows:

\begin{theorem}\label{Giroux-2}
Let $M$ be a closed orientable three-manifold and $\tilde{\lambda}$ a
contact form on $M$. Then there exists a smooth function $f:M\rightarrow
(0,\infty)$ so that the contact form $\lambda=f\tilde{\lambda}$ has a Reeb vector field admitting a global surface of section.
\end{theorem}\qed

Existence results for finite energy foliations with a given contact form $\lambda$ are hard to come by since they usually have striking consequences. In the article \cite{HWZ-sc} for example, H. Hofer, K. Wysocki and E. Zehnder show that every compact strictly convex energy hypersurface $S$ in ${\mathbb R}^4$ carries either two or infinitely many closed characteristics. The proof relies on constructing a special finite energy foliation. In special cases they were established by H. Hofer, K. Wysocki and E. Zehnder \cite{HWZ-fef} and by C. Wendl \cite{Wendl-2}, \cite{Wendl-3}. Proofs usually require a 'starting point', i.e. a finite energy foliation for a slightly different situation as the given one. Then some kind of continuation argument is employed where all kinds of things can and do happen to the original foliation. In \cite{HWZ-fef} the authors start with an explicit finite energy foliation for the round three dimensional sphere $S^3\subset {\mathbb R}^4$ which is then deformed. Chris Wendl's papers also use a rather special manifold as a starting point. The main result of this paper, theorem \ref{main-thm}, provides a 'starting finite energy foliation' for any closed three dimensional contact manifold $(M,\ker\lambda)$ since it is obtained from deforming the leaves of Giroux's open book decomposition. The pages are usually not punctured spheres, and generically there are no pseudoholomorphic curves on punctured surfaces with genus which are transverse to the Reeb vector field. This makes the introduction of the harmonic form in (\ref{eq1}) a necessity. The price to be paid is that compactness issues are more complicated.\\
Chris Wendl \cite{Wendl} published a proof of theorem \ref{main-thm} for the special case where $\ker\lambda$ is a planar contact structure, i.e. the surfaces $\dot{S}$ are punctured spheres. This result was outlined in the article \cite{ACH}. Regardless of whether the contact structure is planar or not, there are two main steps in the proof: Existence of a solution and Compactness of a family of solutions. While the author established the compactness part for theorem \ref{main-thm} long before the article \cite{ACH} appeared we will use the same argument as in Wendl's article \cite{Wendl} for the existence part since it simplifies the proof considerably.\\
The main theorem of this article was the first step in the proof of the Weinstein conjecture for the planar case in \cite{ACH}. Recall that the Weinstein conjecture states the following:\\

\noindent {\bf Conjecture (A. Weinstein, 1978):}\\
{\it Every Reeb vector field $X$ on a
closed contact manifold $M$ admits a periodic orbit.}\\

In fact, Weinstein added the additional hypothesis that the first
cohomology group $H^1(M,{\mathbb R})$ with real coefficients vanishes, but
there seems to be no indication that this additional hypothesis is
needed.\\

Moreover, theorem \ref{main-thm} is also the starting point for the construction of global surfaces of section in the forthcoming paper \cite{AHL-2}. Another application will be an alternative proof of the Weinstein conjecture in dimension three \cite{AHL-2} as outlined in the paper \cite{ACH}. This complements Clifford Taubes' recent proof of the Weinstein conjecture in dimension three using a perturbed version of the Seiberg-Witten equations \cite{Taubes-1}, \cite{Taubes-2}. The main issue with the homological perturbed holomorphic curve equation (\ref{eq1}) is that there is no natural compactification of the space of solutions unless the harmonic forms are uniformly bounded. In the forthcoming papers \cite{AHL-1}, \cite{AHL-2} the lack of compactness is investigated, and bounds for the harmonic forms are derived in particular cases.\\

I am very grateful to Richard Siefring and Chris Wendl for explaining some of their work to me. Their results are indispensable for the arguments in this article. I would also like to thank Samuel Lisi for having numerous discussions with me about the subject of this article.

\section{Existence and local foliations}\label{existence}
\subsection{Local model near the binding orbits}

We will use the same approach as in \cite{Wendl} and \cite{Wendl-2} to prove existence of a solution to (\ref{eq1}). Given a closed contact three manifold $(M,\xi)$, Giroux's theorem implies that there is an open book decomposition as in theorem \ref{Alexanderstheorem} supporting $\xi$. On the other hand, any other contact structure $\xi'$ supported by the same open book is diffeomorphic to $\xi$. Starting with an open book decomposition for $M$, we construct a contact structure supported by it with Giroux contact form $\lambda$ which has a certain normal form near the binding.
\begin{definition}\label{local-model}
 Let $\theta\in S^1={\mathbb R}/2\pi{\mathbb Z}$, denote polar coordinates on the unit disk $D\subset {\mathbb R}^2$ by $(r,\phi)$, and let $\gamma_1,\gamma_2:[0,+\infty)\rightarrow{\mathbb R}$ be smooth functions. A 1-form 
\[
 \lambda\,=\,\gamma_1(r)\,d\theta\,+\,\gamma_2(r)\,d\phi
\]
is called a {\sl local model near the binding} if the following conditions are satisfied:
\begin{enumerate}
 \item The functions $\gamma_1,\gamma_2$ and $\gamma_2(r)/r^2$ are smooth if considered as functions on the disk $D$. In particular, $\gamma_1'(0)=\gamma'_2(0)=\gamma_2(0)=0$.
\item $$\mu(r):=\gamma_1(r)\gamma'_2(r)-\gamma'_1(r)\gamma_2(r)>0\ \mbox{if}\ r>0$$
\item $$\gamma_1(0)>0\ \mbox{and}\ \gamma'_1(r)<0\  \mbox{if}\ r>0$$
\item $$\lim_{r\rightarrow 0}\frac{\mu(r)}{r}=\gamma_1(0)\gamma_2''(0)>0$$
\item $$\kappa:=\frac{\gamma_1''(0)}{\gamma_2''(0)}\notin{\mathbb Z}\ \mbox{and}\ \kappa\le-\frac{1}{2}$$ 
\item $$A(r)=\frac{1}{\mu^2(r)}(\gamma''_2(r)\gamma_1'(r)-\gamma_1''(r)\gamma_2'(r))$$ is of order $r$ for small $r>0$.
\end{enumerate}
\end{definition}

We explain some of the conditions above. First, since
\[
 \lambda\wedge d\lambda=\mu(r)d\theta\wedge dr\wedge d\phi=\frac{\mu(r)}{r}d\theta\wedge dx\wedge dy
\]
the form $\lambda$ is a contact form on $S^1\times D$. The Reeb vector field is given by
\[
 X(\theta,r,\phi)=\frac{\gamma'_2(r)}{\mu(r)}\frac{\partial}{\partial\theta}-\frac{\gamma'_1(r)}{\mu(r)}\frac{\partial}{\partial\phi}=:\alpha(r)\frac{\partial}{\partial\theta}+\beta(r)\frac{\partial}{\partial\phi}.
\]
The trajectories of $X$ all lie on tori $T_r=S^1\times \partial D_r$:
\begin{equation}\label{reebflow}
 \theta(t)=\theta_0+\alpha(r)\,t\ ,\ \phi(t)=\phi_0+\beta(r)\,t.
\end{equation}
We compute
\[
 \lim_{r\rightarrow 0}\alpha(r)=\lim_{r\rightarrow 0}\frac{\gamma_2''(r)}{\mu'(r)}=\frac{\gamma''_2(0)}{\gamma_1(0)\gamma''_2(0)}=\frac{1}{\gamma_1(0)}
\]
and
\[
 \lim_{r\rightarrow 0}\beta(r)=-\lim_{r\rightarrow 0}\frac{\gamma''_1(r)}{\mu'(r)}=-\frac{\gamma''_1(0)}{\gamma_1(0)\gamma''_2(0)}
\]
Recalling that $\frac{\partial}{\partial\phi}=x\frac{\partial}{\partial y}-y\frac{\partial}{\partial x}$ we obtain for $r=0$
\[
 X=\frac{1}{\gamma_1(0)}\frac{\partial}{\partial\theta}
\]
i.e. the central orbit has minimal period $2\pi\gamma_1(0)$. If the ratio $\alpha(r)/\beta(r)$ is irrational then the torus $T_r$ carries no periodic trajectories. Otherwise, $T_r$ is foliated with periodic trajectories of minimal period
\[
 \tau=\frac{2\pi m}{\alpha}=\frac{2\pi n}{\beta}
\]
where $\alpha/\beta=m/n$ or $\beta/\alpha=n/m$ for suitable integers $m,n$ (choose whatever makes sense if either $\alpha$ or $\beta$ is zero). We calculate
\begin{eqnarray*}
 \lim_{r\rightarrow 0}\frac{d\alpha}{dr} & = & \lim_{r\rightarrow 0}\frac{\gamma''_2(r)\mu(r)-\gamma'_2(r)\mu'(r)}{\mu^2(r)}\\
& = & \lim_{r\rightarrow 0}\frac{\gamma_2'''(r)}{2\mu'(r)}-\lim_{r\rightarrow 0}\frac{\mu''(r)}{2\mu'(r)}\,\frac{\gamma'_2(r)}{r}\,\frac{r}{\mu(r)}\\
& = & \frac{\gamma_2'''(0)}{2\mu'(0)}-\frac{\gamma_1(0)\gamma_2'''(0)\gamma_2''(0)}{2(\mu'(0))^2}\\
& = & 0
\end{eqnarray*}
since $\mu''(0)=\gamma_1(0)\gamma_2'''(0)$ and $\mu'(0)=\gamma_1(0)\gamma_2''(0)>0$. Converting to Cartesian coordinates on the disk we get
\[
 X(\theta,x,y)=\alpha(x,y)\frac{\partial}{\partial\theta}-\beta(x,y)\,y\,\frac{\partial}{\partial x}\,+\,\beta(x,y)\,x\,\frac{\partial}{\partial y},
\]
and linearizing the Reeb vector field along the center orbit yields
\[
 DX(\theta,0,0)=\left(\begin{array}{ccc} 0 & 0 & 0 \\ 0 & 0 & -\beta(0)\\ 0 & \beta(0) & 0\end{array}\right)
\]
The linearization of the Reeb flow is given by
\begin{equation}\label{lin-flow}
 D\phi_t(\theta,0,0)=\left(\begin{array}{ccc} 1 & 0 & 0 \\ 0 & \cos\beta(0)t & -\sin\beta(0)t\\ 0 & \sin\beta(0)t & \cos\beta(0)t \end{array}\right)
\end{equation}
with
\[
\Phi(t)=e^{\beta(0)tJ}\ ,\ J=\left(\begin{array}{cc} 0 & -1 \\ 1 & 0 \end{array}\right).
\]
the spectrum of $\Phi(t)$ is giving by
\[
 \sigma(\Phi(t))=\{e^{\pm i\beta(0)t}\}.
\]
The binding orbit has period $2\pi \gamma_1(0)$. Since
\[
\gamma_1(0)\beta(0)\,=\,-\frac{\gamma''_1(0)}{\gamma''_2(0)}\notin {\mathbb Z}
\]
it is nondegenerate and elliptic. 

\begin{example}\label{local-example-1}
 For the contact form $T\,d\theta+\frac{1}{k}(x\,dy-y\,dx)=T\,d\theta+\frac{r^2}{k}d\phi$ the central orbit $S^1\times \{0\}$ is degenerate, but 
\[
 \lambda=(1-r^2)(T\,d\theta+\frac{r^2}{k}d\phi)
\]
is a local model near the binding if
\[
 k,T>0\ ,\ kT\notin{\mathbb Z}\ \mbox{and}\ kT\ge\frac{1}{2}
\]
In this case
\[
 \mu(r)=\frac{2rT}{k}(1-r^2)^2>0\ \mbox{and}\ \frac{\gamma''_1(0)}{\gamma''_2(0)}=-kT,
\]
and we note that
\[
 A(r)=\frac{1}{\mu^2(r)}(\gamma''_2(r)\gamma_1'(r)-\gamma_1''(r)\gamma_2'(r))=\frac{4kr}{T(1-r^2)^4}.
\]
If
\[
 \frac{\alpha(r)}{\beta(r)}=-\frac{\gamma'_2(r)}{\gamma'_1(r)}=\frac{1-2r^2}{kT}=\frac{m}{n}
\]
for integers $n,m$ then the invariant torus $T_r$ is foliated with periodic orbits. The case $m=0$ is only possible if $r=\frac{1}{\sqrt{2}}$. If $r$ is sufficiently small then $|m|\ge 2$. Indeed, we would otherwise be able to find sequences $r_l\searrow 0$ and $\{n_l\}\subset{\mathbb Z}$ such that $kT/(1-2r_l^2)\,=\,n_l$ which is impossible. The binding orbit has period $2\pi T$ while the periodic orbits close to the binding orbit have much larger periods equal to
\[
 \tau\,=\,2\pi Tm\frac{(1-r^2)^2}{1-2r^2}.
\]

\end{example}
\begin{example}\label{local-example-2}
Consider the contact form $\lambda=T(1-r^2)d\theta+\frac{r^2}{k}d\phi$ on $S^1\times D$. It is also a local model near the binding if $k,T>0$, $kT\ge\frac{1}{2}$ and $kT$ is not an integer. We even have $A(r)\equiv 0$. In contrast to the previous example, if $kT\notin{\mathbb Q}$ the invariant tori $T_r$ carry no periodic orbits. If $kT=\frac{n}{m}\in{\mathbb Q}$, but not in ${\mathbb Z}$, all invariant tori are foliated with periodic orbits of period $2\pi mT$ with $|m|\ge 2$ while the binding orbit has period $2\pi T$. The function $A(r)$ is identically zero. This is the contact form on the 'irrational ellipsoid' in ${\mathbb R}^4$. 
\end{example}

The following proposition is essentially proposition 1 from \cite{Wendl}. The construction in the proof was used by Thurston and Winkelnkemper \cite{TW} to show existence of contact forms on closed three manifolds.
\begin{proposition}\label{localnormalform}
Let $M$ be a three dimensional manifold given by an open book decomposition 
\[
 M=W(h)\,{\bigcup}_{\mbox{Id}}\,(\partial W\times
D^2)
\]
as described in theorem \ref{Alexanderstheorem}. We denote the pages by
\[
 F_{\alpha}:=(W\times\{\alpha\})\,{\bigcup}_{\mbox{Id}}\,(\partial W\times I_{\alpha})\ ,\ 0\le\alpha<2\pi,
\]
where $I_{\alpha}:=\{re^{i\alpha}\in D\,|\,0<r< 1\}$, and the binding $\partial W\times\{0\}$ by $K$. Moreover, let $\lambda_2$ be a contact form on $\partial W\times D$ which is a local model near the binding on each connected component of $\partial W\times D$.\\
Then there is a smooth family of 1-forms $(\lambda_{\delta})_{0\le\delta<1}$ on $M$ such that 
\begin{itemize}
\item The form $\lambda_0$ is a confoliation 1-form, i.e. $\lambda_0\wedge d\lambda_0\ge 0$, and $\ker\lambda_0$ agrees with the tangent spaces to the pages $F_{\vartheta}$ away from the binding
\item For $\delta>0$ the forms $\lambda_{\delta}$ are contact forms such that $\ker\lambda_{\delta}$ is supported by the above open book. In particular, the Reeb vector fields $X_{\lambda_{\delta}}$ are transverse to the pages $F_{\alpha}$, and the binding $K$ consists of periodic orbits of $X_{\lambda_{\delta}}$. 
\item The forms $\lambda_{\delta}$ agree with the local model $\lambda_2$ near the binding. In particular, the binding orbits are nondegenerate and elliptic.
\end{itemize}
\end{proposition}

\begin{proof}
We will first construct contact/confoliation forms $\lambda_1$ on $W(h)$, depending smoothly on a parameter $\delta\ge 0$, that we control well near the boundary $\partial W(h)\approx\partial W\times S^1$. Then we will glue these forms together with $\lambda_2$ in a smooth way to obtain a contact form on $W(h){\bigcup}_{\mbox{Id}}(\partial W\times D^2)$ for $\delta>0$ or a confoliation form for $\delta=0$. This procedure was used by Thurston and Winkelnkemper \cite {TW} where they showed that every open book is supported by some contact structure.\\

Starting with an open book as above, we can find a collar neighborhood $C$ of $\partial W$ so that
$h(t,\theta)=(t,\theta)$ for all $(t,\theta)\in C$. Here we
identify $(C,\partial W)$ with $([0,\varepsilon]\times
(\dot{\bigcup}_nS^1),\{0\}\times(\dot{\bigcup}_nS^1))$
 where we take an n--fold disjoint union of circles $S^1\approx{\mathbb R}/2\pi{\mathbb Z}$
  according to the number $n$ of components of
$\partial W$.

We claim that there is an area form $\Omega$ on $W$ that
satisfies
\begin{itemize}
\item $\int_W\Omega=2\pi n$,
\item $\Omega|_C=\,dt\wedge d\theta.$
\end{itemize}
Indeed, start with any area form $\Omega'$ so that
$\int_W\Omega'=2\pi n$. Then we have $\Omega'|_C=f'(t,\theta)dt\wedge d\theta$ with a positive smooth function $f'$ (after switching signs if necessary). Pick now a new
smooth positive function $f$ which is equal to some constant $c$ if
$t\le\frac{1}{3}\varepsilon$ and agrees with $f'$ if
$t\ge\frac{2}{3}\varepsilon$ so that the resulting area form
$\Omega$ still satisfies $\int_W\Omega=2\pi n$. Do one component of $\partial
W$ at a time. Rescaling the t--coordinate we may assume that $c=1$.\\

Let $\alpha_1$ be any 1-form on $W$ which equals $(1+t)\,d\theta$
near $\partial W$. Then we obtain by Stokes' theorem:
$$\int_W(\Omega-d\alpha_1)=2\pi n-\int_{\partial
W}\alpha_1=2\pi n+\int_{\partial W} d\theta =0.$$ The 2--form
$\Omega-d\alpha_1$ on $W$ is closed and vanishes near $\partial
W$. Then there exists a
1-form $\beta$ on $W$ with $$d\beta=\Omega-d\alpha_1$$ and
$\beta\equiv 0$ near $\partial W$. Define now
$\alpha_2:=\alpha_1+\beta$. Then $\alpha_2$ satisfies:
\begin{equation}\label{2eq3}
    d\alpha_2\ \ \mbox{is an area form on $W$ inducing the same orientation as
    $\Omega$},
\end{equation}
\begin{equation}\label{2eq4}
    \alpha_2  =  (1+t)\,d\theta\ \ \mbox{near $\partial W$}.
\end{equation}
The set of 1-forms on $W$ satisfying $(\ref{2eq3})$ and
$(\ref{2eq4})$ is therefore nonempty and also convex. We define
the following 1-form on $W\times[0,2\pi]$, where $\alpha$ is any
1--form on $W$ satisfying $(\ref{2eq3})$ and $(\ref{2eq4})$.
$$\tilde{\alpha}(x,\tau):=\tau\alpha(x)+(2\pi-\tau)(h^{\ast}\alpha)(x).$$
This 1-form descends to the quotient $W(h)$ and the restriction to
each fiber of the fiber bundle
$W(h)\stackrel{\pi}{\rightarrow}S^1$ satisfies condition
$(\ref{2eq3})$. Moreover, since $h\equiv$ Id near $\partial W$ we
have $\tilde{\alpha}(x,\tau)=2\pi(1+t)\,d\theta$ for all
$(x,\tau)=((t,\theta),\tau)$ near $\partial W(h)=\partial W\times
S^1$.\\ Let $d\tau$ be a volume form on $S^1$. We claim that
$$\lambda_1:=-\delta\tilde{\alpha}+\pi^{\ast}d\tau$$ are contact forms
on $W(h)$ whenever $\delta> 0$ is sufficiently small. Pick $(x,\tau)\in W(h)$ and let $\{u,v,w\}$ be a
basis of $T_{(x,\tau)}W(h)$ with $\pi_{\ast}u=\pi_{\ast}v=0$. Then
\begin{eqnarray*}
    & & (\lambda_1\wedge d\lambda_1)(x,\tau)(u,v,w)\\
    & = & \delta^2(\tilde{\alpha}\wedge
    d\tilde{\alpha})(x,\tau)(u,v,w)\,-\,\delta\,[d\tau(\pi_{\ast}w)\,d\tilde{\alpha}(x,\tau)(u,v)]\\
     & \neq & 0
\end{eqnarray*}
for sufficiently small $\delta> 0$, and $d\lambda_1$ is a volume form on $W$. Now we have to continue the contact
forms $\lambda_1$ beyond $\partial W(h)\approx \partial W\times
S^1$ onto $\partial W\times D^2$. At this point it is convenient
to change coordinates. We identify $C\times S^1$ with $\partial
W\times (D^2_{1+\varepsilon}\backslash D^2_{1})$, where $D^2_\rho$
is the 2-disk of radius $\rho$. Using polar coordinates $(r,\phi)$
on $D^2_{1+\varepsilon}$ with $0\le\phi\le 2\pi$ and
$0<r\le1+\varepsilon$, our old coordinates are related to the new
ones by $$\partial W\times (D^2_{1+\varepsilon}\backslash
D^2_1)\ni (\theta,r,\phi)\approx(\theta,1+t,\tau)\in C\times S^1$$
and $\lambda_1$ is given by $$\lambda_1=-\frac{\delta}{2\pi}\,r\,d\theta +\,d\phi$$ on
$\partial W\times (D^2_{1+\varepsilon}\backslash D^2_{1})$, with
$\varepsilon$ sufficiently small so that (\ref{2eq4}) holds. We will from now on drop the factor $1/2\pi$, absorbing it into the constant $\delta$. We have to extend this now smoothly to a contact form on $\partial
W\times D^2_{1+\varepsilon}$ which agrees with $\lambda_2$ near $\{r=0\}$.
We set 
\[
 \lambda\,=\,\gamma_1(r)\,d\theta\,+\,\gamma_2(r)\,d\phi
\]
where $\gamma_1,\gamma_2$ satisfy the conditions in definition \ref{local-model} for small $r$, say $r\le \varepsilon_0$, and
\[
 \gamma_1(r)=-\delta r\ \ ,\ \ \gamma_2(r)= 1 \ \  \mbox{for $r\ge 1-\varepsilon_0$}
\]
If we write $\gamma(r)=\gamma_1(r)+i\gamma_2(r)=\rho(r)\,e^{i\alpha(r)}$ then $$\mu(r):=\gamma_1(r)\gamma'_2(r)-\gamma'_1(r)\gamma_2(r)=\Re(i\gamma(r)\,\overline{\gamma'(r)})=\rho^2(r)\alpha'(r)$$ 
which has to be positive. Also recall from definition \ref{local-model} that
$$\gamma_1(0)>0\ \mbox{and}\ \gamma'_1(r)<0\  \mbox{if}\ r>0$$
hence the curves $\gamma=\gamma_{\delta}$ have to turn counterclockwise in the first quadrant starting at the point $(\gamma_1(0),0)$ and later connecting with $(-\delta(1-\varepsilon_0),1)$. In the case where $\delta> 0$ the Reeb vector fields are given by
\[
 X_{\delta}(\theta,r,\phi)=\frac{\gamma_2'(r)}{\mu(r)}\frac{\partial}{\partial\theta}-\frac{\gamma'_1(r)}{\mu(r)}\frac{\partial}{\partial\phi},
\]
and in particular
\begin{equation}\label{Reeb-vfs}
 X_{\delta}(\theta,r,\phi)=  \frac{\partial}{\partial\phi} \  \mbox{for}\ r\ge 1-\varepsilon_0
\end{equation}
which implies that the Reeb vector fields $X_{\delta}$ converge as $\delta\searrow 0$. In addition to $\lambda_{\delta}$ being contact forms for $\delta>0$  we also want the given open book decomposition to support $\ker\lambda_{\delta}$, hence $X_{\delta}$ needs to be transverse to the pages of the open book decomposition which is equivalent to $\gamma'_1(r)\neq 0$. A curve $\gamma(r)$ fulfilling these conditions can clearly be constructed. 
\end{proof}

The following result shows that we can always assume that a Giroux contact form is equal to any of the forms provided by \ref{localnormalform}.
\begin{proposition}\label{localnormalform-2}
Let $M$ be a closed three dimensional manifold with contact structure $\xi$. Then for every $\delta> 0$ there is a diffeomorphism $\varphi_{\delta}:M\rightarrow M$ such that $\ker\lambda_{\delta}=\varphi_{\ast}\xi$ where $\lambda_{\delta}$ is given by proposition \ref{localnormalform}. 
\end{proposition}

\begin{proof}
 Existence of an open book decomposition supporting $\xi$ follows from the existence part of Giroux's theorem. On the other hand, proposition \ref{localnormalform} yields contact forms $\lambda_{\delta}$ such that $\ker\lambda_{\delta}$ is also supported by the same open book decomposition as $\xi$ for any $\delta> 0$. By the uniqueness part of Giroux's theorem, $\xi$ and $\ker\lambda_{\delta}$ are diffeomorphic. 
\end{proof}
It follows from our previous construction of the forms $\lambda_{\delta}$ that $\lambda_0$ satisfies $\lambda_0\wedge d\lambda_0>0$ on $\partial W\times D_{1-\varepsilon_0}$, and $\lambda_0=d\phi$ otherwise. For $\delta\rightarrow 0$ the Reeb vector fields $X_{\delta}$ will converge to some vector field $X_0$ which is the Reeb vector field of $\lambda_0$ if $r<1-\varepsilon_0$ and which equals $\frac{\partial}{\partial\phi}$ everywhere else.

\begin{proposition}\label{first-solution}
Let $M$ be a closed three dimensional manifold with an open book decomposition and a family of 1-forms $\lambda_{\delta}$, $\delta\ge 0$ as in proposition \ref{localnormalform}. Then there are
\begin{itemize}
 \item a smooth family $(\tilde{J}_{\delta})_{\delta\ge 0}$ of almost complex structures on $T({\mathbb R}\times M)$ which are ${\mathbb R}$-independent, which satisfy $\tilde{J}_{\delta}(X_{\delta})=-\partial/\partial\tau$ where $\tau$ denotes the coordinate on ${\mathbb R}$, so that $J_{\delta}:=\tilde{J}_{\delta}|_{\ker\lambda_{\delta}}$ are $d\lambda_{\delta}$-compatible whenever $\lambda_{\delta}$ is a contact form
\item A parametrization of the Giroux leaves $u_{\alpha}:\dot{S}\rightarrow M$, $\alpha\in[0,2\pi]$, where $\dot{S}=S\backslash\{p_1,\ldots,p_n\}$, and $S$ is a closed surface.
\item A smooth family of smooth functions $a_{\alpha}:\dot{S}\rightarrow{\mathbb R}$
\end{itemize}
such that $\tilde{u}_{\alpha}=(a_{\alpha},u_{\alpha}):\dot{S}\rightarrow{\mathbb R}\times M$ is a family of embedded $\tilde{J}_0$-holomorphic curves for a suitable smooth family of complex structures $j_{\alpha}$ on $S$ which restrict to the standard complex structure on the cylinder $[0,+\infty)\times S^1$ after introducing polar coordinates near the punctures. Moreover, all the punctures are positive\footnote{A puncture $p_j$ is called positive for the curve $(a_{\alpha},u_{\alpha})$ if $\lim_{z\rightarrow p_j}a_{\alpha}(z)=+\infty$.} and the family $(\tilde{u}_{\alpha})_{0\le\alpha\le 2\pi}$ is a finite energy foliation, and the curves $\tilde{u}_{\alpha}$ are $\tilde{J}_{\delta}$-holomorphic near the punctures.
\end{proposition}
\begin{proof}
We parameterize the leaves of the open book decomposition $u_{\alpha}:\dot{S}\rightarrow M$, $0\le\alpha<2\pi$, and we assume that they look as follows near the binding
\begin{equation}\label{Giroux-leaves-near-binding}
\begin{array}{c}
u_{\alpha}:[0,+\infty)\times S^1\longrightarrow S^1\times D_1\\
 u_{\alpha}(s,t)=\left(t,r(s)e^{i\alpha}\right)\end{array}
\end{equation}
where $r$ are smooth functions with $\lim_{s\rightarrow\infty}r(s)=0$ to be determined shortly. We use the notation $(r,\phi)$ for polar coordinates on the disk $D=D_1$. We identify some neighborhood $U$ of the punctures of $\dot{S}$ with a finite disjoint union of half-cylinders $[0,+\infty)\times S^1$. Recall that the binding orbit is given by
\[
 x(t)=\left(\frac{t}{\gamma_1(0)},0,0\right)\ ,\ 0\le t \le 2\pi\gamma_1(0)
\]
and it has minimal period $T=2\pi\gamma_1(0)$. We define smooth functions $a_{\alpha}:\dot{S}\rightarrow{\mathbb R}$ by
\[
 a_{\alpha}(z):=\left\{\begin{array}{cc} \int_0^s\gamma_1(r(s'))\,ds'&\mbox{if}\ z=(s,t)\in[0,+\infty)\times S^1\subset U\\ 0&\mbox{if}\ z\notin U\end{array}\right.
\]
so that
\[
 u_{\alpha}^{\ast}\lambda_0\circ j\,=\,da_{\alpha},
\]
where $j$ is a complex structure on $\dot{S}$ which equals the standard structure $i$ on $[0,+\infty)\times S^1$, i.e. near the punctures. We want to turn the maps $\tilde{u}_{\alpha}=(a_{\alpha},u_{\alpha}):\dot{S}\rightarrow {\mathbb R}\times M$ into $\tilde{J}_0$-holomorphic curves for a suitable almost complex structure $\tilde{J}_0$ on ${\mathbb R}\times M$. Recall that the contact structure is given by
\[
\ker\lambda_{\delta}=\,\mbox{Span}\{\eta_1,\eta_2\}=\,\mbox{Span}\left\{\frac{\partial}{\partial r},-\gamma_2(r)\frac{\partial}{\partial \theta}+\gamma_1(r)\frac{\partial}{\partial\phi}\right\}.
\]
We define complex structures $J_{\delta}:\ker\lambda_{\delta}\rightarrow \ker\lambda_{\delta}$ by
\begin{equation}\label{cplx-str-normal-form}
 J_{\delta}(\theta,r,\phi)\left(-\gamma_2(r)\frac{\partial}{\partial \theta}+\gamma_1(r)\frac{\partial}{\partial\phi}\right):=-\frac{1}{h(r)}\,\frac{\partial}{\partial r}
\end{equation}
and
\[
 J_{\delta}(\theta,r,\phi)\frac{\partial}{\partial r}:=h(r)\left(-\gamma_2(r)\frac{\partial}{\partial \theta}+\gamma_1(r)\frac{\partial}{\partial\phi}\right)
\]
where $h:(0,1]\rightarrow {\mathbb R}\backslash\{0\}$ are suitable smooth functions. Also recall that $\gamma_1,\gamma_2$ depend on $\delta$ away from the binding orbit. We want $J_{\delta}$ to be compatible with d$\lambda_{\delta}$, i.e. 
\[
 d\lambda_{\delta}(\eta_1,J\eta_1)=h(r)\,\mu(r)>0\ \mbox{and}\ d\lambda_{\delta}(\eta_2,J\eta_2)=\frac{\mu(r)}{h(r)}>0
\]
so that $h(r)>0$. We also demand that $J_{\delta}$ extends smoothly over the binding $\{r=0\}$. Expressing the vectors $\eta_1$ and $\eta_2$ in Cartesian coordinates, we have
\[
 \eta_1=\frac{1}{r}\left(x\,\frac{\partial}{\partial x}+y\,\frac{\partial}{\partial y}\right)
\]
and
\[
 \eta_2=-\gamma_2(r)\frac{\partial}{\partial\theta}+\gamma_1(r)\,x\,\frac{\partial}{\partial y}-\gamma_1(r)\,y\,\frac{\partial}{\partial x}.
\]
We introduce the following generators of the contact structure:
\begin{eqnarray*}
 \varepsilon_1 & := & \gamma_1(r)\frac{\partial}{\partial y}-\frac{x\gamma_2(r)}{r^2}\frac{\partial}{\partial\theta}\\
& = & \frac{y\gamma_1(r)}{r}\eta_1\,+\,\frac{x}{r^2}\eta_2
\end{eqnarray*}
and
\begin{eqnarray*}
 \varepsilon_2 & := & \gamma_1(r)\frac{\partial}{\partial x}+\frac{y\gamma_2(r)}{r^2}\frac{\partial}{\partial\theta}\\
& = & \frac{x\gamma_1(r)}{r}\eta_1\,-\,\frac{y}{r^2}\eta_2.
\end{eqnarray*}
We compute from this
\[
 \eta_1=\frac{1}{r\gamma_1(r)}(y\,\varepsilon_1+x\,\varepsilon_2)\ ,\ \eta_2=x\,\varepsilon_1-y\,\varepsilon_2.
\]
Now
\begin{eqnarray*}
 J_{\delta}\varepsilon_1 & = & \frac{y\gamma_1(r)h(r)}{r}\eta_2-\frac{x}{r^2h(r)}\eta_1\\
& = & \left(\frac{1}{r}xy\gamma_1(r)h(r)-\frac{xy}{r^3h(r)\gamma_1(r)}\right)\,\varepsilon_1-\\
& & -\left(\frac{1}{r}y^2\gamma_1(r)h(r)+\frac{x^2}{r^3\gamma_1(r)h(r)}\right)\,\varepsilon_2
\end{eqnarray*}
and 
\begin{eqnarray*}
 J_{\delta}\varepsilon_2 & = & \frac{x\gamma_1(r)h(r)}{r}\eta_2+\frac{y}{r^2h(r)}\eta_1\\
& = & \left(-\frac{1}{r}xy\gamma_1(r)h(r)+\frac{xy}{r^3h(r)\gamma_1(r)}\right)\,\varepsilon_2\\
& & +\left(\frac{1}{r}x^2\gamma_1(r)h(r)+\frac{y^2}{r^3\gamma_1(r)h(r)}\right)\,\varepsilon_1.
\end{eqnarray*}
Inserting $x=r\cos\phi$, $y=r\sin\phi$ and demanding for a $\phi$-independent limit as $r\rightarrow 0$ we arrive at the condition that $r\,h(r)\,\gamma_1(r)\,\equiv\,\pm\,1$ for small $r$. Recalling that we need $h>0$ we obtain
\[
 h(r)=\frac{1}{r\,\gamma_1(r)}\ \mbox{for small $r$}.
\]
As usual, we continue $J_{\delta}$ to an almost complex structure $\tilde{J}_{\delta}$ on ${\mathbb R}\times M$ by setting
\[
 \tilde{J}_{\delta}(\theta,r,\phi)\frac{\partial}{\partial\tau}:=X_{\delta}(\theta,r,\phi)
\]
where $\tau$ denotes the coordinate in the ${\mathbb R}$-direction. We emphasize that $\tilde{J}_{\delta}$ also makes sense for $\delta=0$. We will now arrange $r(s)$ in (\ref{Giroux-leaves-near-binding}) such that the Giroux leaves $\tilde{u}_{\alpha}=(a_{\alpha},u_{\alpha})$ become $\tilde{J}_0$-holomorphic curves\footnote{The calculation shows that we can make them $J_{\delta}$-holomorphic for all $\delta\ge 0$ near the binding.}.
We compute for $r\le 1-\varepsilon_0$
\begin{eqnarray*}
\partial_s \tilde{u}_{\alpha}+\tilde{J}_0(u_{\alpha})\partial_t\tilde{u}_{\alpha} & = & \gamma_1(r)\frac{\partial}{\partial\tau}+r'\frac{\partial}{\partial r}+\tilde{J}_0(u_{\alpha})\left(\frac{\partial}{\partial\theta}\right)\\
& = & \gamma_1(r)\frac{\partial}{\partial\tau}+r'\frac{\partial}{\partial r}+\tilde{J}_0(u_{\alpha})(\gamma_1(r)\,X_{\delta}(u_{\alpha}))+\\
& & +\tilde{J}_0(u_{\alpha})\left(\frac{\partial}{\partial\theta}-\gamma_1(r)X_{\delta}(u_{\alpha})\right)\\
& = & r'\frac{\partial}{\partial r}+\tilde{J}_0(u_{\alpha})\left(\frac{\gamma_1'(r)}{\mu(r)}\left(\gamma_1(r)\frac{\partial}{\partial\phi}-\gamma_2(r)\frac{\partial}{\partial\theta}\right)\right)\\
& = & \left(r'-\frac{\gamma_1'(r)}{\mu(r)h(r)}\right)\frac{\partial}{\partial r}
\end{eqnarray*}
hence the Giroux leaves satisfy the equation if we choose $r$ to be a solution of the ordinary differential equation
\[
r'(s)=\frac{\gamma_1'(r(s))}{\mu(r(s))\,h(r(s))}.
\]
Note that $r'(s)<0$. We choose also $h(r)\equiv 1$ for $r\ge 1-\varepsilon_0$. We continue the almost complex structures $J_{\delta}:\ker\lambda_{\delta}\rightarrow\ker\lambda_{\delta}$ (which were only defined near the binding) smoothly to all of $M$. Away from the binding we have $X_{\delta}=\partial/\partial\phi$, and we extend $J_{\delta}$ as before to $T({\mathbb R}\times M)$. Away from the binding, if $\delta=0$, we have that $\ker\lambda_0$ coincides with the tangent spaces of the pages of the open book decomposition. Because $a_{\alpha}$ is constant away from the binding, the solutions $\tilde{u}_{\alpha}$ which we constructed near the binding fit together smoothly with the pages of the open book decomposition and solve the holomorphic curve equation for the almost complex structure $J_0$.  
\end{proof}

\begin{remark}\label{decay-of-r}
Near the binding orbit the function $r(s)$ satisfies a differential equation of the form 
\[
 r'(s)=\Lambda(r(s))\,r(s):=\frac{\gamma'_1(r(s))\gamma_1(r(s))}{\mu(r(s))}\,r(s)
\]
and
\[
 \lim_{r\rightarrow 0}\Lambda(r)=\frac{\gamma''_1(0)}{\gamma_2''(0)}=:\kappa.
\]
Writing $r(s)=c(s)e^{\kappa s}$ the function $c(s)$ satisfies $c'(s)=(\Lambda(r(s))-\kappa)c(s)$, hence it is a decreasing function which converges to a constant as $s\rightarrow+\infty$.\\

 We will return to examples \ref{local-example-1} and \ref{local-example-2} and compute $r(s)$ for large $s$. The differential equation in the case of example \ref{local-example-2} for large $s$ is
\[
 r'(s)\,=\,\frac{\gamma'_1(r(s))\gamma_1(r(s))}{\mu(r(s))}r(s)\,=\,-kT(1-r^2(s))r(s),
\]
so that
\[
 r(s)=\frac{1}{\sqrt{1+c\,e^{2kTs}}},
\]
where $c$ is a constant. In example \ref{local-example-1} the differential equation reads
\[
 r'(s)\,=\,\frac{\gamma'_1(r(s))\gamma_1(r(s))}{\mu(r(s))}r(s)\,=\,-\frac{kT}{1-r^2(s)}r(s),
\]
and solutions satisfy
\[
 r(s)=c\,e^{-kT\,s}\,e^{\frac{1}{2}r^2(s)}.
\]
\end{remark}

\subsection{Functional Analytic set-up and the Implicit Function Theorem}

In the following theorem we will prove the existence of a smooth family of solutions near a given solution. In proposition \ref{first-solution} we constructed a finite energy foliation for the data $(\lambda_0,J_0)$ with vanishing harmonic form. The form $\lambda_0$ however is only a confoliation form. We will produce solutions for the perturbed data $(\lambda_{\delta},J_{\delta})$, and harmonic forms will appear if the surface $S$ is not a sphere. The key result is an application of the implicit function theorem in a suitable setting. 

\begin{theorem}\label{IFT}
Assume one of the following:
\begin{enumerate}
\item Let $(a_0,u_0):\dot{S}\rightarrow{\mathbb R}\times M$ be one of the $\tilde{J}_0$-holomorphic curves described in proposition \ref{first-solution} with complex structure $j_0$ on $S$ (we refer to such $u_0$ as a 'Giroux leaf') and confoliation form $\lambda_0$ or
\item let $(\dot{S},j_0,a_0,u_0,\gamma_0)$ be a solution of the differential equation (\ref{eq1}) for some $d\lambda_0$-compatible complex structure $J_0:\ker\lambda_0\rightarrow\ker\lambda_0$ which, near the binding orbit, agrees with (\ref{cplx-str-normal-form}), and where $\lambda_0$ is a contact form which is a local model near the binding. Assume that $u_0$ is an embedding and that it is of the form $u_0=\phi_g(v_0)$ where $g:S\rightarrow {\mathbb R}$ is a smooth function, $\phi$ is the flow of the Reeb vector field and where $v_0:\dot{S}\rightarrow M$ is a Giroux leaf as in proposition \ref{first-solution}
\end{enumerate}
Let $J_{\delta}$ be a smooth family of $d\lambda_{\delta}$-compatible complex structures also agreeing with (\ref{cplx-str-normal-form}) near the binding orbit, where $(\lambda_{\delta})_{-\varepsilon<\delta<+\varepsilon}$, $\varepsilon>0$ is a smooth family of 1-forms which are contact forms for $\delta\neq 0$ and local models near the binding. Then there is a smooth family $$(S,j_{\delta,\tau},a_{\delta,\tau},u_{\delta,\tau},\gamma_{\delta,\tau},J_{\delta})_{-\varepsilon<\delta,\tau<+\varepsilon}$$ of solutions of equation (\ref{eq1}) so that $u_{\delta,\tau}(\dot{S})\cap u_{\delta,\tau'}(\dot{S})=\emptyset$ whenever $\tau\neq \tau'$ and each $u_{\delta,\tau}$ is an embedding.
\end{theorem}

\begin{proof}
 In both cases we wish to find solutions of (\ref{eq1}) for the data $(\lambda_{\delta},J_{\delta})$ of the form
\[
 u_{\delta}(z)\,=\,\phi_{f_{\delta}(z)}(u_0(z))\ \ ,\ \ a_{\delta}(z)\,=\,b_{\delta}(z)+a_0(z)
\]
where $t\mapsto \phi_t=\phi^{\delta}_t$ is the flow of the Reeb vector field $X_{\delta}$ of $\lambda_{\delta}$, and where $b_{\delta}+if_{\delta}:S\rightarrow{\mathbb C}$ is a smooth function defined on the unpunctured surface. We will derive an equation for the unknown function $b_{\delta}+if_{\delta}$. From now we will suppress the superscript $\delta$ in the notation unless for $\delta=0$. Because of the first equation in (\ref{eq1}) the complex structure on $S$ is then determined by $f$, denote it by $j=j_f$, and it is given by
\begin{equation}\label{implicit-eq1}
 j_f(z)  =  (\pi_{\lambda} Tu(z))^{-1}\circ J(u(z))\circ \pi_{\lambda} Tu(z) 
\end{equation}
Note that this is well defined because $u$ is transverse to the Reeb vector field so that $\pil Tu(z):T_zS\rightarrow \ker\lambda(u(z))$ is an isomorphism.
By the second equation of (\ref{eq1}), we then have to solve the equation $df \circ j_f+u_0^{\ast}\lambda\circ j_f=da+\gamma$ for $a,f,\gamma$ on $\dot{S}$ which is equivalent to the equation
\begin{equation}\label{IFT-eq1}
\bar{\partial}_{j_f}(a+if)=u_0^{\ast}\lambda\circ j_f-i(u_0^{\ast}\lambda)-\gamma-i(\gamma\circ j_f).
\end{equation}
Recalling that we are looking for $a$ of the form $a=a_0+b$ where $b$ is a suitable real valued function defined on the whole surface $S$. We obtain the differential equation
\begin{equation}\label{IFT-eq2}
\bar{\partial}_{j_f}(b+if)=u_0^{\ast}\lambda\circ j_f-i(u_0^{\ast}\lambda)-\bar{\partial}_{j_f}a_0-\gamma-i(\gamma\circ j_f).
\end{equation}
and it follows from a straight forward calculation (see appendix \ref{local-computations}) that all expressions on the right hand side of equation (\ref{IFT-eq2}) are bounded near the punctures, in particular they are contained in the spaces $L^p(T^{\ast}S\otimes {\mathbb C})$ for any $p$. This is what the assumption $\kappa\le-\frac{1}{2}$ from definition \ref{local-model} is needed for. We will work in the function space $b+if\in W^{1,p}(S,\C)$ where $p>2$. For any complex structure $j$ on $S$ the space $L^{p}(T^{\ast}S\otimes\Complex)$ of complex valued 1--forms of class $L^p$ decomposes into complex linear and complex antilinear forms (with respect to $j$). We use the notation
\[
L^{p}(T^{\ast}S\otimes\Complex)=L^{p}(T^{\ast}S\otimes\Complex)^{1,0}_j\oplus L^{p}(T^{\ast}S\otimes\Complex)^{0,1}_j.
\]
The operator $b+if\mapsto \bar{\partial}_{j_f}(b+if)$ is then a section in the vector bundle
\[
L^{p}(T^{\ast}S\otimes\Complex)^{0,1}:=\bigcup_{b+if\in W^{1,p}(S,\Complex)}\{b+if\}\times L^{p}(T^{\ast}S\otimes\Complex)^{0,1}_{j_f}\rightarrow W^{1,p}(S,\Complex).
\]
This vector bundle is of course trivial, but here are some explicit local trivializations for $f\,,\,g\in W^{1,p}(S,\R)$ sufficiently close to each other: 
\begin{equation}\label{IFT-eq3}
\Psi_{fg}:L^{p}(T^{\ast}S\otimes\Complex)^{0,1}_{j_f}\widetilde{\longrightarrow} L^{p}(T^{\ast}S\otimes\Complex)^{0,1}_{j_g}
\end{equation}
\[
\tau\longmapsto \tau+i(\tau\circ j_g).
\]
If we write 
\[
\tau+i(\tau\circ j_g)=\tau\circ(\mbox{Id}_{TS}\,-j_f\circ j_g),
\]
we see that $\Psi_{fg}$ is invertible with
\[
\Psi_{fg}^{-1}\tau=\tau\circ(\mbox{Id}_{TS}\,-j_f\circ j_g)^{-1}.
\]
It follows from the Hodge decomposition theorem that every cohomology class $[\sigma]\in H^1(S,\R)$ has a unique harmonic representative $\psi_j(\sigma)\in {\mathcal H}^1_j(S)$ where ${\mathcal H}^1_{j}(S)$ is defined as 
\begin{equation}
{\mathcal H}_j^1(S):=\{\gamma\in {\mathcal E}^1(S)\,|\,d\gamma=0\,,\,d(\gamma\circ j)=0\}
\end{equation}
and where ${\mathcal E}^1(S)$ denotes the space of all (smooth) real valued 1--forms on $S$, and we write ${\mathcal E}^{0,1}(S)={\mathcal E}_j^{0,1}(S)$ for the space of complex antilinear 1--forms on $S$ with respect to $j$, i.e. complex valued 1--forms $\sigma$ such that $i\,\sigma+\sigma\,j=0$. Note that our definition coincides with the set of closed and co-closed 1-forms on $S$. Moreover, by elliptic regularity, we may also consider Sobolev forms. We will identify $H^1(S,\R)$ with $\R^{2g}$, and we consider the following parameter dependent section in the bundle $L^p(T^{\ast}S\otimes \C)^{0,1}\rightarrow W^{1,p}(S,\C)$
\begin{equation}\label{nonlinear-op}
F:W^{1,p}(S,\C)\times \R^{2g}\longrightarrow L^p(T^{\ast}S\otimes\C)^{0,1}
\end{equation}
\begin{eqnarray*}
F(b+if,\sigma) & := & \bar{\partial}_{j_f}(b+if)-u_0^{\ast}\lambda\circ j_f+i(u_0^{\ast}\lambda)+\\
& & +\bar{\partial}_{j_f}a_0+\psi_{j_f}(\sigma)+i(\psi_{j_f}(\sigma)\circ j_f)
\end{eqnarray*}
with $j_f$ as in (\ref{implicit-eq1}). Recalling that $z\mapsto j_f(z)$ may not be differentiable, we interpret the equation $d(\gamma\circ j_f)=0$ in the sense of weak derivatives. The solution set of (\ref{IFT-eq2}) is then the zero set of $F$. We consider the real parameter $\delta$ which we dropped from the notation, fixed at the moment. For $g\equiv 0$ and $b+if$ small in the $W^{1,p}$--norm we consider the composition $\hat{F}(b+if,\sigma)=\Psi_{fg}(F(b+if,\sigma))$. Its linearization in the point $(b+if,\sigma)=(0,\sigma_0)$ where $\sigma_0$ is defined by $\psi_{j_0}(\sigma_0)=\gamma_0$, and where $F(0,\sigma_0)=0$, is
\[
D\hat{F}(0,\sigma_0):W^{1,p}(S,\C)\times\R^{2g}\longrightarrow L^p(T^{\ast}S\otimes\C)^{0,1}_{j_0}
\]
\begin{eqnarray*}
D\hat{F}(0,\sigma_0)(\zeta,\sigma) & = & \bar{\partial}_{j_0}\zeta+\psi_{j_0}(\sigma)+i(\psi_{j_0}(\sigma)\circ j_0)+L\,\zeta,
\end{eqnarray*}
where 
\[
L:W^{1,p}(S,\C)\rightarrow W^{1,p}(T^{\ast}S\otimes\C)^{0,1}_{j_0}\hookrightarrow L^p(T^{\ast}S\otimes\C)^{0,1}_{j_0}
\]
is a compact linear map\footnote{The linear map $\zeta\mapsto L\zeta$ only depends on the imaginary part of $\zeta$. It is actually given by 
\[
L\zeta = -\frac{1}{2}u_0^{\ast}\lambda\circ(A\zeta+j_0A\zeta j_0)+\frac{i}{2}u_0^{\ast}\lambda\circ(j_0A\zeta-A\zeta j_0)+B\zeta+i\,B\zeta\,j_0
\]
where
\[
B\zeta=\left.\frac{d}{d\tau}\right|_{\tau=0}\psi_{j_{\tau\,k}}(\sigma_0)\ ,\ \zeta=h+ik
\]
and
\[
A\zeta  = \left.\frac{d}{d\tau}\right|_{\tau=0}j_{\tau\,k}=\,h\,(\pil Tu_0)^{-1}\Big[J(u_0)DX_{\lambda}(u_0)-DX_{\lambda}(u_0)J(u_0)+DJ(u_0)X_{\lambda}(u_0)\Big](\pil Tu_0).
\]
} because we are working on a compact domain $S$. The linear term $L$ therefore does not contribute to the Fredholm--index of $D\hat{F}(0,\sigma_0)$. We claim that the operator 
\[
W^{1,p}(S,\C)\times\R^{2g}\longrightarrow L^p(T^{\ast}S\otimes\C)^{0,1}_{j_0}
\]
\[
(\zeta,\sigma)\longmapsto \bar{\partial}_{j_0}\zeta+\psi_{j_0}(\sigma)+i(\psi_{j_0}(\sigma)\circ j_0)
\]
is a surjective Fredholm operator of index two. Then we would have $\,\mbox{ind}(D\hat{F}(0,\sigma_0))=2$ as well. Here is the argument: The Riemann Roch theorem asserts that the kernel and the cokernel of the Cauchy-Riemann operator $\bar{\partial}_j$ (acting on smooth complex valued functions on $S$) are both finite dimensional and that
\[
\mbox{dim}_{{\Real}}\ker\bar{\partial}_j-\mbox{dim}_{{\Real}}({\mathcal E}^{0,1}(S)/\mbox{Im}\,\bar{\partial}_j)=2-2g,
\]
where $g$ is the genus of the surface $S$. The only holomorphic functions on $S$ are the constant functions, hence ${\mathcal E}^{0,1}(S)/\mbox{Im}\,\bar{\partial}_j$ has dimension $2g$.\\
On the other hand, the vector space ${\mathcal H}_j^1(S)$ of all (real--valued) harmonic 1--forms on $S$ also has dimension $2g$ (see \cite{Forster}). We consider now the linear map
\[
\Psi:{\mathcal H}_j^1(S)\longrightarrow {\mathcal E}^{0,1}(S)/\mbox{Im}\,\bar{\partial}_j
\]
\[
\Psi(\gamma):=[\gamma+i(\gamma\circ j)],
\]
where $[\,.\,]$ denotes the equivalence classes of $(0,1)$--forms. Assume $\Psi(\gamma)=[0]$, i.e. there is a complex--valued smooth function $f=u+iv$ on $S$ such that $\bar{\partial}_jf=\gamma+i(\gamma\circ j)$. Since $\gamma$ is a harmonic 1--form, we conclude that $d(dv\circ j)=d(du\circ j)=0$, i.e. both $u$ and $v$ are harmonic. Since there are only constant harmonic functions on $S$ we obtain $\gamma=0$, i.e. $\Psi$ is injective and also bijective. Hence $(0,1)$--forms $\gamma+i(\gamma\circ j)$ with $\gamma\in{\mathcal H}_j^1(S)$ make up the cokernel of $\bar{\partial}_j:C^{\infty}(S,{\mathbb C})\rightarrow {\mathcal E}^{0,1}(S)$.\\ 
This proves the claim that the operator $D\hat{F}(0,\sigma_0)$ is Fredholm of index two. We will now show that the operator $D\hat{F}(0,\sigma_0)$ is surjective. Using the decomposition
\[
L^p(T^{\ast}S\otimes\C)^{0,1}_{j_0}=R(\bar{\partial}_{j_0})\oplus{\mathcal H}^1_{j_0}(S)
\]
and denoting the corresponding projections by $\pi_1,\pi_2$ we see that it suffices to prove surjectivity of the operator
\[
T:W^{1,p}(S,\C)\rightarrow R(\bar{\partial}_{j_0})
\]
\[
T\zeta:=\bar{\partial}_{j_0}\zeta+\pi_1(L\zeta)
\]
which is a Fredholm operator of index $2$. Assume $\zeta\in\ker T$. Unless $\zeta\equiv 0$ the set $\{z\in S\,|\,\zeta(z)=0\}$ consists of finitely many points by the Similarity Principle \cite{HZ-book} and the local degree of each zero is positive. On the other hand, the sum of all the local degrees has to be zero, hence elements in the kernel of $T$ are nowhere zero. Actually, if $h+ik\in\ker T$ then even $k$ is nowhere zero because $h+c+ik\in \ker T$ for any real constant $c$ since the zero order term $L$ only depends on the imaginary part of $\zeta$. Therefore\footnote{Indeed, otherwise we would be able to find three linearly independent elements in the kernel $\zeta_1,\zeta_2,\zeta_3$. Because ${\mathbb C}$ has real dimension two we can find real numbers $\alpha_1,\alpha_2,\alpha_3$, not all simultaneously zero, and a point $z\in S$ such that $\sum_{j=1}^3\alpha_j\zeta_j(z)\,=\,0$. Then $\zeta=\sum_{j=1}^3\alpha_j\zeta_j$ is in the kernel of $T$ and $\zeta(z)=0$, a contradiction.} $\mbox{dim}\,\ker T\le 2$, and since the Fredholm index of $T$ equals $2$, we actually have $\mbox{dim}\,\ker T=2$. This proves surjectivity of $T$ and also of $D\hat{F}(0,\sigma_0)$ so that the set ${\mathcal M}$ of all pairs $(b+if,\gamma)$ solving the differential equation (\ref{IFT-eq2}) is a two dimensional manifold with $T_{(0,\gamma_0)}{\mathcal M}=\ker D\hat{F}(0,\gamma_0)$. If we add a real constant to $b+if$ then we obtain again a solution of (\ref{IFT-eq2}). If we divide ${\mathcal M}$ by this $\R$--action then we obtain a one--dimensional family of solutions $(\tu_{\tau})_{-\varepsilon<\tau<\varepsilon}$ with $\tu_{\tau}=(a_{\tau},u_{\tau})$ for which $u_{\tau}=\phi_{f_{\tau}}(u_0)$, and the functions $f_{\tau}$ do not vanish at any point. Therefore, we have $u_0(\dot{S})\cap u_{\tau}(\dot{S})=\emptyset$ and also $u_{\tau'}(\dot{S})\cap u_{\tau}(\dot{S})=\emptyset$ if $\tau\neq \tau$. Moreover, the maps $u_{\tau}$ are transverse to the Reeb vector field by construction.
\end{proof}

\section{From local foliations to global ones}\label{compactness-section}
The aim of this section is to show that a family of solutions produced by the implicit function theorem (theorem \ref{IFT}) can be enlarged further. For this purpose a compactness result is needed for which we are setting the stage now.\\ 
First, we will summarize a result by Richard Siefring (theorem 2.2 in \cite{Siefring}) which will be used later on:
\begin{theorem}\label{relative-asymptotics}
Let $\tilde{u}\in{\mathcal M}(P,J)$ and $\tilde{v}\in{\mathcal M}(P,J)$, let maps $U,V:[R,\infty)\times S^1\rightarrow C^{\infty}(P^{\ast}\xi)$ be asymptotic representatives of $\tilde{u}$ and $\tilde{v}$, respectively, and assume that $U-V$ does not vanish identically. Then there exists a negative eigenvalue $\lambda$ of the asymptotic operator ${\bf A}_{P,J}$ and an eigenvector $e$ with eigenvalue $\lambda$ so that
\[
 U(s,t)-V(s,t)=e^{\lambda s}(e(t)+r(s,t))
\]
where the map $r$ satisfies for every $(i,j)\in{\mathbb N}^2$, a decay estimate of the form
\[
 |\nabla^i_s\nabla^j_tr(s,t)|\le M_{ij}e^{-ds}
\]
with $M_{ij}$ and $d$ positive constants.
\end{theorem}\qed

Our situation is less general than in \cite{Siefring}, so we will explain the notation in the context of this paper. The setup is a manifold $M$ with contact form $\lambda$ and contact structure $\xi=\ker\lambda$. Consider a periodic orbit $\bar{P}$ of the Reeb vector field $X_{\lambda}$ with period $T$, and we may assume here that $T$ is its minimal period. We introduce $P(t):=\bar{P}(Tt/2\pi)$ such that $P(0)=P(2\pi)$. If $J:\xi\rightarrow\xi$ is a $d\lambda$-compatible complex structure the set of all $\tilde{J}$-holomorphic half-cylinders 
\[
 \tilde{u}=(a,u):[R,\infty)\times S^1\rightarrow {\mathbb R}\times M\ ,\ S^1={\mathbb R}/2\pi{\mathbb Z}
\]
for which $|a(s,t)-Ts/2\pi|$ and $|u(s,t)-P(t)|$ decay at some exponential rate (in local coordinates near the orbit $P(S^1)$) is denoted by ${\mathcal M}(P,J)$. Note that it is assumed here that the domain $[R,\infty)\times S^1$ is endowed with the standard complex structure. A smooth map $U:[R,\infty)\times S^1\rightarrow P^{\ast}\xi$ for which $U(s,t)\in\xi_{P(t)}$ is called an asymptotic representative of $\tilde{u}$ if there is a proper embedding $\psi:[R,\infty)\times S^1\rightarrow{\mathbb R}\times S^1$ asymptotic to the identity so that
\[
 \tilde{u}(\psi(s,t))=(Ts/2\pi,\exp_{P(t)}U(s,t))\ \forall\ (s,t)\in[R,\infty)\times S^1
\]
($\exp$ is the exponential map corresponding to some metric on $M$, for example the one induced by $\lambda$ and $J$). Every $\tilde{u}\in{\mathcal M}(P,J)$ has an asymptotic representative (see \cite{Siefring}). The asymptotic operator ${\bf A}_{P,J}$ is defined as follows:
\[
( {\bf A}_{P,J}h)(t):=-\frac{T}{2\pi}\,J(P(t))\left.\left(\frac{d}{ds}\right|_{s=0}D\phi_{-s}(\phi_s(P(t)))h(\phi_s(P(t)))\,\right)
\]
where $\phi_s$ is the flow of the Reeb vector field, and where $h$ is a section in $P^{\ast}\xi\rightarrow S^1$. Because the Reeb flow preserves the splitting $TM={\mathbb R}\,X_{\lambda}\oplus \xi$ we have also $( {\bf A}_{P,J}h)(t)\in\xi_{P(t)}$.\\
We compute the asymptotic operator ${\bf A}_{P,J}$ for the binding orbit $$\bar{P}(t)=\left(\frac{t}{\gamma_1(0)},0,0\right)\in S^1\times {\mathbb R}^2.$$
Recall that the above periodic orbit has minimal period $T=2\pi\gamma_1(0)$. Using 
\[
 \phi_s(P(t))=\phi_s(t,0,0)=\left(t+\frac{s}{\gamma_1(0)},0,0\right),
\]
formula (\ref{lin-flow}) for the linearization of the Reeb flow with $h(t)=(0,\zeta(t),\eta(t))$ and the fact that $J(t,0,0)\frac{\partial}{\partial x}=\frac{\partial}{\partial y}$ and $J(t,0,0)\frac{\partial}{\partial y}=-\frac{\partial}{\partial x}$ we compute
\begin{eqnarray*}
 ({\bf A}_{P,J}h)(t) & = & -\gamma_1(0)J(t,0,0)\left(\frac{h'(t)}{\gamma_1(0)}+\left(\begin{array}{cc} 0 & \beta(0)\\ -\beta(0) & 0 \end{array}\right)\left(\begin{array}{c} \zeta(t)\\ \eta(t)\end{array}\right)\right)\\
& = & -J_0\,h'(t)-\gamma_1(0)\beta(0)\,h(t)\\
& = & -J_0\,h'(t)+\kappa\,h(t)
\end{eqnarray*}
where $J_0=\left(\begin{array}{cc}0 & -1\\ 1 & 0\end{array}\right)$ and $\kappa=\gamma_1''(0)/\gamma_2''(0)\in(-1,0)$. Hence $\lambda\in\sigma({\bf A}_{P,J})$ precisely if
\[
 h'(t)\,=\,(\lambda-\kappa)\,J_0\,h(t)\ \mbox{and}\ h(2\pi)=h(0),
\]
i.e. $\sigma({\bf A}_{P,J})=\{\kappa+l\,|\,l\in{\mathbb Z}\}$, and the largest negative eigenvalue is given by $\kappa$. The corresponding eigenspace consists of all constant vectors $h(t)\equiv const\in{\mathbb R}^2$. The eigenspace for the eigenvalues $\kappa+l$ consists of all
\[
 h(t)=e^{J_0l\,t}h_0\ \mbox{with}\ h_0\in{\mathbb R}^2.
\]

\begin{theorem}\label{compactness-result}
{\bf (Compactness)}\\
Let $\lambda$ be a contact form on $M$ which is a local model near the binding (of the Giroux leaf $v_0$), and let $J:\ker\lambda\rightarrow\ker\lambda$ be a $d\lambda$-compatible complex structure. Consider a smooth family of solutions $(S,j_{\tau},a_{\tau},u_{\tau},\gamma_{\tau},J)_{0\le\tau<\tau_0}$ to equation (\ref{eq1}) satisfying the following conditions:
\begin{itemize}
\item $u_{\tau}=\phi_{f_{\tau}}(v_0)$ where $v_0:\dot{S}\rightarrow M$ is a Giroux leaf as in proposition \ref{first-solution} and $f_{\tau}:S\rightarrow{\mathbb R}$ are suitable smooth functions.
\item For any $0\le\tau<\tau_0$ there is $\delta>0$ such that
\[
 u_{\tau}(\dot{S})\cap u_{\tau'}(\dot{S})\,=\,\emptyset\ \mbox{whenever}\ 0<|\tau-\tau'|<\delta.
\]
\item Assume that $u_0$ and $u_{\tau}$ never have identical images whenever $0<\tau<\tau_0$.
\end{itemize}
Then the functions $f_{\tau}$ converge uniformly with all derivatives to a smooth function $f_{\tau_0}:S\rightarrow {\mathbb R}$ as $\tau\nearrow \tau_0$. The harmonic 1--forms $\gamma_{\tau}$ also converge in $C^{\infty}(S)$ to a 1--form $\gamma_{\tau_0}$ which is harmonic with respect to the complex structure $j_{\tau_0}$ on $\dot{S}$ given by   
\[
j_{\tau_0}(z):=(\pil Tu_{\tau_0}(z))^{-1}\circ J(u_{\tau_0}(z))\circ \pil Tu_{\tau_0}(z),
\]
where $u_{\tau_0}:=\phi_{f_{\tau_0}}(v_0)$. Moreover, we can find a smooth function $a_{\tau_0}$ on $\dot{S}$ so that $(S,j_{\tau_0},a_{\tau_0},u_{\tau_0},\gamma_{\tau_0},J)$ solves the differential equation (\ref{eq1}).
\end{theorem}

\begin{remark}
 We may assume without loss of generality that $v_0\equiv u_0$ and $f_0\equiv 0$. If $z\in\dot{S}$ we denote by $T(z)>0$ the positive return time of the point $u_0(z)$, i.e.
\[
 T(z):=\inf\{T>0\,|\,\phi_T(u_0(z))\in u_0(\dot{S})\}<+\infty.
\]
We claim that the return time $z\mapsto T(z)$ extends continously over the punctures of the surface, and therefore there is an upper bound
\[
T:=\sup_{z\in\dot{S}}T(z)<\infty
\]
Using (\ref{reebflow}) and (\ref{Giroux-leaves-near-binding}), we note that, asymptotically near the punctures, $\phi_T(u_0(s,t))\in S^1\times {\mathbb R}^2$ has the following structure:
\[
 \phi_T(u_0(s,t))  =  \big(t+\alpha(r(s))T\,,\,r(s)\exp\big[i(\alpha_0+\beta(r(s))T)\big]\big)
\]
where $r(s)$ is a strictly decreasing function, $\alpha_0$ is some constant, and $\alpha(r),\beta(r)$ are suitable functions for which the limits $\lim_{r\rightarrow 0}\beta(r)$ and $\lim_{r\rightarrow 0}\alpha(r)$ exist and are not zero. Hence, if $T=T(u_0(s,t))$ is the positive return time at the point $u_0(s,t)$, then 
\[
T(u_0(s,t))= \frac{2\pi}{|\beta(r(s))|},
\]
and therefore the limit for $s\rightarrow+\infty$ exists.
\end{remark}

The remainder of this section is devoted to the proof of theorem \ref{compactness-result}.
We recall that the functions $a_{\tau}$ and $f_{\tau}$ satisfy the Cauchy-Riemann type equation (\ref{IFT-eq1}) which is
\[
\bar{\partial}_{j_{\tau}}(a_{\tau}+if_{\tau})=u_0^{\ast}\lambda\circ j_{\tau}-i(u_0^{\ast}\lambda)-\gamma_{\tau}-i(\gamma_{\tau}\circ j_{\tau}),
\]
where the complex structure $j_{\tau}$ is given by (\ref{implicit-eq1}) or
\begin{eqnarray*}
 j_{\tau}(z)& = & (\pil Tu_0(z))^{-1}(T\phi_{f_{\tau}(z)}(u_0(z)))^{-1}\cdot  \\
 & & \cdot J(\phi_{f_{\tau}(z)}(u_0(z)))\,T\phi_{f_{\tau}(z)}(u_0(z))\pil Tu_0(z)
\end{eqnarray*}
and $\gamma_{\tau}$ is a closed 1--form on $S$ with $d(\gamma_{\tau}\circ j_{\tau})=0$.\\

 The following $L^{\infty}$-bound is the crucial ingredient for the compactness result: We claim that
\begin{equation}\label{Linfty-bound}
\sup_{0\le\tau<\tau_0}\|f_{\tau}\|_{L^{\infty}(\dot{S})}\le T.
\end{equation}

Restricting any of the solutions to a simply connected subset $U\subset \dot{S}$ we can write $\gamma_{\tau}=dh_{\tau}$ for a suitable function $h_{\tau}:U\rightarrow{\mathbb R}$, and the maps
\[
 \tilde{u}_{\tau}:U\rightarrow {\mathbb R}\times M\ ,\ \tilde{u}_{\tau}=(a_{\tau}+h_{\tau},u_{\tau})
\]
are $\tilde{J}$-holomorphic curves. If two such curves $\tilde{u}_{\tau}$ and $\tilde{u}_{\tau'}$ have an isolated intersection then the corresponding intersection number is positive (see \cite{MW} or \cite{AH}, \cite{McDS} for positivity of (self)intersections for holomorphic curves). We claim that
\[
 u_0(\dot{S})\cap u_{\tau}(\dot{S})\,=\,\emptyset\ \ \forall\ \ 0<\tau<\tau_0
\]
and not just for small $\tau$ as assumed. If we can show this then (\ref{Linfty-bound}) follows. Indeed, for any $z\in\dot{S}$ the function $\tau\mapsto f_{\tau}(z)$ is strictly increasing from $f_0(z)=0$, and equality $f_{\tau}(z)=T(z)$ would imply that $u_{\tau}(z)\in u_0(\dot{S})$. Arguing indirectly, we assume that the set
\[
{\mathcal O}:= \{\tau\in(0,\tau_0)\,|\,u_{\tau}(\dot{S})\cap u_0(\dot{S})\neq\emptyset\}
\]
is not empty. We denote its infimum by $\tilde{\tau}$ which must be a positive number since $u_{\tau}(\dot{S})\cap u_0(\dot{S})=\emptyset$ for all sufficiently small $\tau>0$.\\
We will first prove that the above set is open which implies that $u_{\tilde{\tau}}$ and $u_0$ can not intersect. If $u_{\tau}(p)=u_0(q)$ for suitable points $p,q\in\dot{S}$ then we consider locally near these points the corresponding holomorphic curves $\tilde{u}_{\tau}$ and $\tilde{u}_0$. Adding some constant to the ${\mathbb R}$-component of one of them we may assume that $\tilde{u}_{\tau}(p)=\tilde{u}_0(q)$. If this intersection point is not isolated then $p$ and $q$ have open neighborhoods $U$ and $V$ respectively on which the holomorphic curves $\tilde{u}_{\tau}$ and $\tilde{u}_0$ agree. This implies that the set of all points $p\in\dot{S}$ such that $\tilde{u}_{\tau}(p)$ is a non-isolated intersection point between $\tilde{u}_{\tau}$ and $\tilde{u}_0$, is open and closed, i.e. it is either empty or all of $\dot{S}$. Since we assumed that each set $u_{\tau}(\dot{S})$, $\tau>0$ is different from $u_0(\dot{S})$ we conclude that if $u_{\tau}$ and $u_0$ intersect then the intersection point of the corresponding holomorphic curves $\tilde{u}_{\tau}$ and $\tilde{u}_0$ must be isolated.
But this implies on the other hand that $u_{\tau'}$ and $u_0$ would also intersect for all $\tau'$ sufficiently close to $\tau$ by positivity of the intersection number showing that the set ${\mathcal O}$ is open.\\
We conclude from the above that there are a sequence $\tau_k\searrow\tilde{\tau}$ and points $p_k,q_k\in\dot{S}$ such that $u_{\tau_k}(p_k)=u_0(q_k)$. Passing to a suitable subsequence we may assume convergence of the sequences $(p_k)_{k\in{\mathbb N}}$ and $(q_k)_{k\in{\mathbb N}}$ to points $p,q\in S$. Because of $u_{\tilde{\tau}}(\dot{S})\cap u_0(\dot{S})=\emptyset$ the points $p,q$ must be punctures, and they have to be equal $z_0=p=q\in S\backslash\dot{S}$. The reason for this is the following: The maps $u_{\tau_k}, u_0$ are asymptotic near the punctures to a disjoint union of finitely many periodic Reeb orbits which are not iterates of other periodic orbits. Also, different punctures always correspond to different periodic orbits. This follows from E. Giroux's result and our constructions in section \ref{existence} of this paper.\\ 
We now derive a contradiction using Richard Siefring's result. The harmonic forms $\gamma_{\tau_k}$ in equation (\ref{eq1}) are defined on all of $S$, hence they are exact on some open neighborhood $U$ of the puncture $z_0$ and $\gamma_{\tau_k}=dh_{\tau_k}$ for suitable functions $h_{\tau_k}$ on $U$ and similarly $\gamma_{\tilde{\tau}}=dh_{\tilde{\tau}}$. We may also assume that $j_{\tilde{\tau}}|_U=j_{\tau_k}|_U=j_0$ after changing local coordinates near $z_0$. Then, on the set $U$, the maps $\tilde{u}_{\tau_k}=(a_{\tau_k}+h_{\tau_k},u_{\tau_k})$ and $\tilde{u}_0=(a_0+h_0,u_0)$ are holomorphic curves with $\tilde{u}_{\tau_k}(p_k)=\tilde{u}_0(q_k)$ while the images of $\tilde{u}_{\tilde{\tau}}$ and $\tilde{u}_0$ have empty intersection. Let now 
\[
 U_{\tilde{\tau}}\,,\,U_{\tau_k}\,,\,U_0:[R,\infty)\times S^1\rightarrow {\mathbb R}^2
\]
be asymptotic representatives of the holomorphic curves $\tilde{u}_{\tilde{\tau}}\,,\,\tilde{u}_{\tau_k}\,,\,\tilde{u}_0$ respectively. Invoking theorem \ref{relative-asymptotics} and our subsequent computation of the asymptotic operator and its spectrum we obtain the following asymptotic formulas
\begin{equation}\label{as-formula}
 U_{\tau}(s,t)-U_0(s,t)\,=\,e^{\lambda_{\tau}\,s}(e_{\tau}(t)+r_{\tau}(s,t))\ ,\ \tau=\tilde{\tau},\tau_k\,,\,\,s\ge R_{\tau}
\end{equation}
where $R_{\tau}>0$ is some constant, $\lambda_{\tau}< 0$ is some negative eigenvalue of the asymptotic operator ${\bf A}_{P,J}$. It is of the form $\lambda_{\tau}=\kappa+l_{\tau}$ where $l_{\tau}$ is an integer, $\kappa=\gamma_1''(0)/\gamma_2''(0)$ is not an integer, and where $e_{\tau}(t)=e^{J_0l_{\tau}t}h_{\tau}$, $h_{\tau}\in{\mathbb R}^2\backslash\{0\}$ is an eigenvector corresponding to the eigenvalue $\lambda_{\tau}=\kappa+l_{\tau}$. Note that the above formula applies since $U_{\tau}-U_0$ can not vanish identically. We will actually show that $l_{\tau}\equiv 0$. The asymptotic representative $U_0$ is given by
\[
u_0(s,t)=(t,r(s)e^{i\alpha_0})=(t, U_0(s,t)),
\]
using equation (\ref{Giroux-leaves-near-binding}), and we recall that $r(s)=c(s)e^{\kappa s}$ where $c(s)\rightarrow c_{\infty}>0$ as $s\rightarrow+\infty$. An asymptotic representative of $\tilde{u}_{\tau}$ is however given by an expression such as
\[
 u_{\tau}(\psi(s,t))=(t,U_{\tau}(s,t))
\]
where $\psi:[R,\infty)\times S^1\rightarrow{\mathbb R}\times S^1$ is a proper embedding converging to the identity map as $s\rightarrow+\infty$. Writing $(s',t')=\psi(s,t)$ we get using formula equations (\ref{reebflow}) for the Reeb flow
\begin{eqnarray*}
 U_{\tau}(s,t) & = & c(s')e^{\kappa s'}e^{i(\alpha_0+\beta(r(s'))f_{\tau}(s',t'))}\\
& = & e^{\kappa s}(e_{\tau}+r_{\tau}(s,t)).
\end{eqnarray*}
The asymptotic formula for $U_{\tau}$ apriori allows for other decay rates but $\kappa$ is the only possible one. Dividing by $e^{\kappa s}$ and passing to the limit $s\rightarrow+\infty$ we obtain
\[
 e_{\tau}=c_{\infty}e^{i\alpha_0}e^{i\beta(0)f_{\tau}(\infty)}
\]
where $f_{\tau}(\infty)=\lim_{s\rightarrow+\infty}f_{\tau}(s,t)$ which is independent of $t$ since $f_{\tau}$ extends continuously over the punctures. Hence the difference $U_{\tau}-U_0$ has decay rate $\lambda_{\tau}\equiv\kappa$ as claimed unless the two eigenvectors $e_{\tau}$ and $e_0$ agree which is equivalent to 
\[
 f_{\tau}(\infty)\in \frac{2\pi}{\beta(0)}\,{\mathbb Z}
\]
or $\tau=\tilde{\tau}$ in our case. The maps $U_{\tau}-U_0$ satisfy a Cauchy-Riemann type equation to which the Similarity Principle applies so that for every zero $(s,t)$ of $U_{\tau}-U_0$ the map $\sigma\mapsto (U_{\tau}-U_0)(s+\epsilon\cos \sigma,t+\epsilon\sin\sigma)$ has positive degree for small $\epsilon>0$. The Cauchy-Riemann type equation mentioned above is derived in \cite{Siefring} in section 5.3 as well as in section 3 of \cite{Abbas-2} in a slightly different context and also in \cite{HWZ-3}.
If $R$ is sufficiently large then the map
\[
S^1\rightarrow S^1\ ,\ t\mapsto W_{\tau}(R,t):=\frac{U_{\tau}-U_0}{|U_{\tau}-U_0|}(R,t) 
\]
is well-defined, and it has degree $l_{\tau}$ because the remainder term $r_{\tau}(s,t)$ decays exponentially in $s$. Zeros of $U_{\tau}-U_0$ contribute in the following way: If $R'<R$ such that $(U_{\tau}-U_0)(R',t)\neq 0$ then
\begin{equation}\label{zero-formula}
 \mbox{deg}W_{\tau}(R,\cdot)\,=\,\mbox{deg}W_{\tau}(R',\cdot)+\sum_{\{z|U_{\tau}(z)-U_0(z)=0\}}o(z).
\end{equation}
We know already that $l_{\tau}=0$ whenever $\tau\neq \tilde{\tau}$. Arguing indirectly, we assume that $l_{\tilde{\tau}}$ is not zero. It would have to be negative then. Choose then $R'>0$ so large that $\mbox{deg}W_{\tilde{\tau}}(R',\cdot)=l_{\tilde{\tau}}<0$. For $\tau$ sufficiently close to $\tilde{\tau}$ we also have $\mbox{deg}W_{\tau}(R',\cdot)=l_{\tilde{\tau}}$. On the other hand we have $\mbox{deg}W_{\tau}(R,\cdot)=0$ for $R>R'$ sufficiently large. Equation (\ref{zero-formula}) implies that the map $U_{\tau}-U_0$ must have zeros in $[R',R]\times S^1$ to account for the difference in degrees, but we know that there are none for $\tau<\tilde{\tau}$. This contradiction shows that $l_{\tilde{\tau}}\neq 0$ is impossible. Choose again $R'>0$ so large that $\mbox{deg}W_{\tilde{\tau}}(R',\cdot)=0$. The degree does not change if we slightly alter $\tau$. In particular, we have $\mbox{deg}W_{\tau}(R',\cdot)=0$ for $\tau>\tilde{\tau}$ close to $\tilde{\tau}$ as well. For $R>>R'$ we have $\mbox{deg}W_{\tilde{\tau}}(R,\cdot)=0$, and we recall that
\[
(U_{\tau_k}-U_0)(s_k,t_k)=0\ ,\ \tau_k\searrow\tilde{\tau}
\]
for a suitable sequence $(s_k,t_k)$ with $s_k\rightarrow+\infty$, and the set of zeros of $U_{\tau_k}-U_0$ is discrete. This however contradicts equation (\ref{zero-formula}) since the zeros have positive orders.\\
Summarizing, we have shown that the assumption ${\mathcal O}\neq\emptyset$ leads to a contradiction which implies the apriori bound (\ref{Linfty-bound}).\\
The monotonicity of the functions $f_{\tau}$ in $\tau$ and the bound (\ref{Linfty-bound}) imply that the functions $f_{\tau}$ converge pointwise to a measureable function $f_{\tau_0}$ as $\tau\nearrow {\tau_0}$. We also know that $\|f_{\tau_0}\|_{L^{\infty}(\dot{S})}\le T$. We then obtain a complex structure $j_{\tau_0}$ on $\dot{S}$ by
\begin{eqnarray*}
 j_{\tau_0}(z)& = & (\pil Tu_0(z))^{-1}(T\phi_{f_{\tau_0}(z)}(u_0(z)))^{-1}\cdot  \\
 & & \cdot J(\phi_{f_{\tau_0}(z)}(u_0(z)))\,T\phi_{f_{\tau_0}(z)}(u_0(z))\pil Tu_0(z).
\end{eqnarray*}
By definition the complex structure $j_{\tau_0}$ is also of class $L^{\infty}$ and $j_{\tau}(z)\rightarrow j_1(z)$ pointwise. Our task is to improve the regularity of the limit $f_{\tau_0}$ and the character of the convergence $f_{\tau}\rightarrow f_{\tau_0}$. We also have to establish convergence of the functions $a_{\tau}$ for $\tau\nearrow {\tau_0}$. The complex structures $j_{\tau}$ are of course all smooth, but the limit $j_{\tau_0}$ might only be measureable.

\subsection{The Beltrami equation}

For the convenience of the reader we briefly summarize a few classical facts from the theory of quasiconformal mappings (see the paper by L. Ahlfors and L. Bers \cite{Ahlforsbook}, \cite{Bers-paper}).\\
The punctured surface $\dot{S}$ carries metrics $g_{\tau}$, also of class $L^{\infty}$ for $\tau=\tau_0$ and smooth otherwise, so that 
\[
g_{\tau}(z)(j_{\tau}(z)v,j_{\tau}(z)w)=g_{\tau}(z)(v,w)\ \mbox{ for all}\ v,w\in T_z\dot{S}.
\]
In fact, $g_{\tau}$ is given by
\[
g_{\tau}(z)(v,w)=d\lambda(u_{\tau}(z))\Big(\pil Tu_{\tau}(z)v,J(u_{\tau}(z))\pil Tu_{\tau}(z)w\Big).
\]
In the case $\tau=\tau_0$ we replace $\pil Tu_{\tau}(z)$ by $T\phi_{f_{\tau_0}(z)}(u_0(z))\pil Tu_0(z)$. We have $\sup_{\tau}\|g_{\tau}\|_{L^{\infty}(\dot{S})}<\infty$ and $g_{\tau}\rightarrow g_{\tau_0}$ pointwise as $\tau\nearrow {\tau_0}$. Our considerations about the regularity of the limit are of local nature, so we may replace $\dot{S}$ with a ball $B\subset {\C}$ centered at the origin. Denoting the metric tensor of $g_{\tau}$ by $(g^{\tau}_{kl})_{1\le k,l\le 2}$ we define the following complex--valued smooth functions:
\[
\mu_{\tau}(z):=\frac{\frac{1}{2}(g^{\tau}_{11}(z)-g^{\tau}_{22}(z))+i\,g^{\tau}_{12}(z)}{\frac{1}{2}(g^{\tau}_{11}(z)+g^{\tau}_{22}(z))+\sqrt{g^{\tau}_{11}(z)g^{\tau}_{22}(z)-(g^{\tau}_{12}(z))^2}},
\]
and we note that
\[
\sup_{\tau}\|\mu_{\tau}\|_{L^{\infty}(\dot{S})}<1
\]
and that $\mu_{\tau}\rightarrow\mu_{\tau_0}$ pointwise. We view the functions $\mu_{\tau}$ as functions on the whole complex plane by trivially extending them beyond $B$. Then they are also $\tau$--uniformly bounded in $L^p({\mathbb C})$ for all $1\le p\le\infty$ and $\mu_{\tau}\rightarrow\mu_{\tau_0}$ in $L^p({\mathbb C})$ for $1\le p<\infty$ by Lebesgue's theorem. If we solve now the Beltrami equation
\[
\overline{\partial}\alpha_{\tau}=\mu_{\tau}\,\partial\alpha_{\tau}
\]
for $\tau<{\tau_0}$ so that $\partial\alpha_{\tau}(0)\neq 0$ then $\alpha_{\tau}$ is a diffeomorphism of the plane onto itself so that
\[
g_{\tau}(\alpha_{\tau}(z))(T\alpha_{\tau}(z)v,T\alpha_{\tau}(z)w)=\lambda_{\tau}\langle v,w\rangle\ \mbox{if}\ z\in B,
\]
where $\langle\,.\,,\,.\,\rangle$ denotes the standard Euclidean scalar product on ${\mathbb R}^2$ and $\lambda_{\tau}$ is a positive function. We then get
\[
T\alpha_{\tau}(z)\circ \,i\,=\,j_{\tau}(\alpha_{\tau}(z))\circ T\alpha_{\tau}(z)\ ,\ 0\le\tau < 1\ \mbox{if}\ z\in B.
\]
For H\"{o}lder continuous $\mu_{\tau}$ the map $\alpha_{\tau}$ exists, and it is a $C^1$--diffeomorphism. This is a classical result by A. Korn and L. Lichtenstein \cite{Lichtenstein}. More modern proofs may be found for example in \cite{Bers-lecture} and \cite{Chern}. In our case we have smooth solutions $\alpha_{\tau}$ belonging to smooth $\mu_{\tau}$, but we only know that the $\mu_{\tau}$ converge pointwise as $\tau\nearrow 1$. On the other hand, we would like to derive a decent notion of convergence for the transformations $\alpha_{\tau}$. An interesting case for us is the one where $\mu$ is only a measureable function. Results in this direction were obtained by C. B. Morrey \cite{Morrey}, L. Ahlfors and L. Bers \cite{Bers-paper} and also by L. Bers and L. Nirenberg \cite{Bers-Nirenberg}. We also refer to the book \cite{Ahlforsbook} by Lars Ahlfors. We summarize now a few results from the paper \cite{Bers-paper} about the Beltrami equation for measureable $\mu$ which we will need later on. The first result concerns the inhomogeneous Beltrami equation
\[
\overline{\partial}u=\mu\,\partial\,u+\sigma,
\]
where $u:{\mathbb C}\rightarrow {\Complex}$, $\mu$ is a complex--valued measureable function on $\Complex$ with
\[
\|\mu\|_{L^{\infty}(\Complex)}<1
\]
and $\sigma\in L^p(\Complex)$ for a suitable $p>2$ (we will explain shortly what values for $p$ are admissible). We consider the following operators acting on smooth functions with compact support in the plane:
\[
(Ag)(z):=\frac{1}{2\pi i}\int_{\Complex}g(\xi)\left(\frac{1}{\xi-z}-\frac{1}{\xi}\right)d\xi\,d\overline{\xi},
\]
\[
(\Gamma g)(z):=\frac{1}{2\pi i}\lim_{\ve\searrow 0}\int_{\Complex\backslash B_{\ve}(0)}\frac{g(\xi)-g(z)}{(\xi-z)^2}d\xi\,d\overline{\xi}.
\]
Both operators can be extended to continuous operators $L^p(\Complex)\rightarrow L^p(\Complex)$ for all $2<p<\infty$. They have the following properties:
\begin{enumerate}
\item $\overline{\partial}(Ag)=A(\overline{\partial} g)=g,$
\item $\partial(A g)=A(\partial g)=\Gamma g,$
\item $|Ag(z_1)-Ag(z_2)|\le\,C_p\|g\|_{L^p(\Complex)}|z_2-z_1|^{1-\frac{2}{p}}$,
\item $\|\Gamma g\|_{L^p(\Complex)}\le \,c_p\,\|g\|_{L^p(\Complex)}$ with $c_p\rightarrow 1$ as $p\rightarrow 2$.
\end{enumerate}
We used here the notation $\overline{\partial}=\frac{1}{2}(\partial_s+i\partial_t)$ and $\partial=\frac{1}{2}(\partial_s-i\partial_t)$. Properties (1) and (2) above should be understood in the sense of distributions. The proof of (4) involves the Calder\'{o}n--Zygmund inequality and the Riesz--Thorin convexity theorem (see \cite{Lax} and \cite{Stein}). Following \cite{Bers-paper} we define $B_p$ with $p>2$, to be the space of all locally integrable functions on the plane which have weak derivatives in $L^p({\C})$, vanish in the origin and which satisfy a global H\"{o}lder condition with exponent $1-\frac{2}{p}$. For $u\in B_p$ we then define a norm by
\[
\|u\|_{B_p}:=\sup_{z_1\neq z_2}\frac{|u(z_2)-u(z_1)|}{|z_2-z_1|^{1-\frac{2}{p}}}+\|\partial u\|_{L^p(\Complex)}+\|\overline{\partial}u\|_{L^p(\Complex)},
\]
so that $B_p$ becomes a Banach space. We will usually choose $p>2$ such that $c_p\,\sup_{\tau}\|\mu_{\tau}\|_{L^{\infty}(\Complex)}<1$ where $c_p$ is the constant from item (4) above.

\begin{theorem}{\bf (see \cite{Bers-paper}, theorem 1)}\label{inhomogeneous-Beltrami}\\
Assume that $p>2$ such that $c_p\,\sup_{\tau}\|\mu_{\tau}\|_{L^{\infty}(\Complex)}<1$. If $\sigma\in L^p(\Complex)$ then the equation
\[
\overline{\partial}u=\mu\,\partial u+\sigma
\]
has a unique solution $u=u_{\mu,\sigma}\in B_p$.
\end{theorem}
For the existence part of the theorem, one first solves the following fixed point problem in $L^p(\Complex)$
\[
q\,=\,\Gamma(\mu\,q)+\Gamma\sigma.
\]
This is possible because the map 
\[
L^p({\C})\longrightarrow L^p({\C})
\]
\[
q\longmapsto \Gamma(\mu q+\sigma)
\]
is a contraction in view of $c_p\|\mu\|_{L^{\infty}(\Complex)}<1$. Then
\[
u:=A(\mu\,q+\sigma)
\]
is the desired solution. The following estimate is also derived in the paper \cite{Bers-paper}
\begin{equation}\label{5.1.2.}
\|q\|_{L^p(\Complex)}\le\,c'_p\|\sigma\|_{L^p(\Complex)},
\end{equation}
with $c'_p=c_p/(1-c_p\|\mu\|_{L^{\infty}(\Complex)})$ which follows from
\begin{eqnarray*}
\|q\|_{L^p({\C})} & \le & \|\Gamma(\mu q)\|_{L^p({\C})}+\|\Gamma \sigma\|_{L^p({\C})}\\
 & \le & c_p\|\mu\|_{L^{\infty}({\C})}\|q\|_{L^p({\C})}+c_p\|\sigma\|_{L^p({\C})}.
\end{eqnarray*}

Recalling our original situation we have the following result which shows that there is some sort of conformal mapping for $j_1$ on the ball $B$.
\begin{theorem}{\bf (see \cite{Bers-paper}, theorem 4)}\label{thm-5.3}\\
Let $\mu:{\C}\rightarrow {\C}$ be an essentially bounded measureable function with $\mu|_{{\C}\backslash B}\equiv 0$ and $p>2$ such that $c_p\|\mu\|_{L^{\infty}({\C})}<1$. Then there is a unique map $\alpha:\Complex\rightarrow\Complex$ with $\alpha(0)=0$ such that
\[
\overline{\partial}\alpha=\mu\partial\alpha
\]
in the sense of distributions with $\partial\alpha-1\in L^p(\Complex)$.
\end{theorem}
The desired map $\alpha$ is given by
\[
\alpha(z)\,=\,z+u(z).
\]
where $u\in B_p$ solves the equation $\overline{\partial}u=\mu\partial u+\mu$. In particular, $\alpha\in W^{1,p}(B)$. Lemma 8 in \cite{Bers-paper} states that $\alpha:\Complex\rightarrow\Complex$ is a homeomorphism. We can apply the theorem to all the $\mu_{\tau}$, $0<\tau\le 1$ and obtain smooth $\mu_{\tau}$--conformal mappings $\alpha_{\tau}$ together with the associated maps $u_{\tau}$. 
\begin{lemma}\label{lemma-6.4}
Let $\mu_n:{\C}\rightarrow{\C}$ be a sequence of measureable functions so that $\mu_n|_{{\C}\backslash B}\equiv 0$ and $\sup_n\|\mu_n\|_{L^{\infty}({\C})}<1$. Assume that $\mu_n\rightarrow\mu$ pointwise almost everywhere. Then the corresponding quasiconformal mappings $\alpha_n,\alpha$ as in theorem \ref{thm-5.3} satisfy
\[
\|\alpha_n-\alpha\|_{W^{1,p}(B)}\longrightarrow 0
\]
as $n\rightarrow\infty$, for any $p>2$ such that $c_p\,\sup_n\|\mu_n\|_{L^{\infty}({\C})}<1$ and any compact set $B\subset {\mathbb C}$.
\end{lemma}
{\bf Proof:}\\
We first estimate with $g\in L^p({\C})$ and $z\neq 0$ 
\begin{eqnarray}\label{Beltrami-convergence-estimate}
|Ag(z)| & = & \frac{1}{2\pi}\left|\int_{{\mathbb C}} g(\xi)\frac{z}{\xi(\xi-z)}\right|\,d\xi\,d\bar{\xi}\nonumber\\
 & \le & \frac{|z|}{2\pi}\|g\|_{L^p({\C})}\left\|\frac{1}{\xi(\xi-z)}\right\|_{L^{\frac{p}{p-1}}({\C})}\\
 & \le & C_p\|g\|_{L^p({\C})}\,|z|^{1-\frac{2}{p}},\nonumber
\end{eqnarray}
where the last estimate holds in view of
\begin{eqnarray*}
\int_{{\mathbb C}} |\xi(\xi-z)|^{-\frac{p}{p-1}}\,d\xi\,d\bar{\xi} & \stackrel{\zeta=z^{-1}\xi}{=} & \int_{{\mathbb C}}|z^2\zeta^2-z^2\zeta|^{-\frac{p}{p-1}}|z|^2\,d\zeta\,d\bar{\zeta}\\
 & = & |z|^{2-\frac{2p}{p-1}}\underbrace{\int_{{\mathbb C}}|\zeta(\zeta-1)|^{-\frac{p}{p-1}}d\zeta\,d\bar{\zeta}}_{2\pi C_p}.
\end{eqnarray*}
If $q$ solves $q=\Gamma(\mu q+\mu)$ then
\begin{eqnarray*}
\bar{\partial}(\alpha_n-\alpha) & = & \mu_n\,\partial(\alpha_n-\alpha)-\mu\,\partial\alpha+\mu_n\,\partial\alpha\\
 & = & \mu_n\,\partial(\alpha_n-\alpha)+\mu_n-\mu+(\mu_n-\mu)\Gamma(\mu q+\mu),
 \end{eqnarray*}
i.e. the difference $\alpha_n-\alpha$ again satisfies an inhomogeneous Beltrami equation. By theorem \ref{inhomogeneous-Beltrami} we have
\[
\alpha_n-\alpha=A(\mu_n q_n+\lambda_n),
\]
where $\lambda_n=\mu_n-\mu+(\mu_n-\mu)\Gamma(\mu q+\mu)$ and where $q_n\in L^p({\C})$ solves $q_n=\Gamma(\mu_n q_n+\lambda_n)$. Combining this with the estimates (\ref{Beltrami-convergence-estimate}) and (\ref{5.1.2.}) we obtain
\begin{eqnarray}\label{estimate-5.1}
|\alpha_n(z)-\alpha(z)| & \le & C_p\,\|\mu_n q_n+\lambda_n\|_{L^p({\C})}\,|z|^{1-\frac{2}{p}}\nonumber\\
 & \le & \Big(C_p\,\sup_n\|\mu_n\|_{L^{\infty}({\C})}\cdot c'_p\|\lambda_n\|_{L^p({\C})}+C_p\,\|\lambda_n\|_{L^p({\C})}\Big)\,|z|^{1-\frac{2}{p}}.
\end{eqnarray}
Since $\|\mu_n-\mu\|_{L^p({\C})}\rightarrow 0$ and $\|(\mu_n-\mu)\Gamma(\mu q+\mu)\|_{L^p({\C})}\rightarrow 0$ by Lebesgue's theorem we also have $\|\lambda_n\|_{L^p({\C})}\rightarrow 0$ and therefore $\alpha_n\rightarrow \alpha$ uniformly on compact sets. Since $\bar{\partial}(\alpha_n-\alpha)=\mu_n\,\partial(\alpha_n-\alpha)+\lambda_n$ and $\alpha_n-\alpha=A(\mu_n q_n+\lambda_n)$ we verify that
\[
\partial(\alpha_n-\alpha)=\Gamma(\mu_n q_n+\lambda_n)=q_n
\]
and
\[
\bar{\partial}(\alpha_n-\alpha)=\mu_n q_n+\lambda_n.
\]
Invoking (\ref{5.1.2.}) once again we see that both $\|\partial(\alpha_n-\alpha)\|_{L^p({\C})}$ and $\|\bar{\partial}(\alpha_n-\alpha)\|_{L^p({\C})}$ can be bounded from above by a constant times $\|\lambda_n\|_{L^p({\C})}$ which converges to zero.\\
 
\qed

We will also need some facts concerning the classical case where $\mu\in C^{k,\alpha}(B_R(0))$, $B_R(0)=\{z\in{\mathbb C}\,|\,|z|<R\}$, which are not spelled out explicitly in \cite{Bers-lecture} or in \cite{Chern}, but which easily follow from the constructions carried out there. 
\begin{theorem}\label{classical-Bers}
Let $\mu,\gamma,\delta\in C^{k,\alpha}(B_{R'}(0))$ with $0<\alpha<1$ and $\sup_{B_{R'}(0)}|\mu|<1$. Then for sufficiently small $0<R\le R'$ there is a unique solution $w\in C^{k+1,\alpha}(B_R(0))$ to the equation
\[
\overline{\partial}w(z)=\mu(z)\partial w(z)+\gamma(z)w(z)+\delta(z)
\]
with $w(0)=0$ and $\partial w(0)=1$. If $w_1,w_2$ solve the above equation with coefficient functions $\mu_l,\gamma_l,\delta_l$, $l=1,2$ then there is a constant $c=c(\alpha,R,\|w_2\|_{C^{k,\alpha}(B_R(0))},k)>0$ such that for all $w_1\in C^{k+1,\alpha}(B_R(0))$
\begin{eqnarray*}
\|w_2-w_1\|_{C^{k+1,\alpha}(B_R(0))}&\le &c\,(\|\delta_2-\delta_1\|_{C^{k,\alpha}(B_R(0))}+\\
 & & +\|\mu_2-\mu_1\|_{C^{k,\alpha}(B_R(0))}+\|\gamma_2-\gamma_1\|_{C^{k,\alpha}(B_R(0))}).
\end{eqnarray*}
\end{theorem}
{\bf Sketch of the proof:}\\
The existence proof is a slight generalization of the Korn--Lichtenstein result (see also \cite{Bers-lecture} or \cite{Chern}). What we are looking for is the estimate. We define the following operator
\[
(Tw)(z):=A(\mu\partial w+\gamma w)(z)-z\,\Gamma(\mu\partial w+\gamma w)(0)
\]
and the function
\[
g(z):=(A\delta)(z)-z\,(\Gamma\delta)(0)+z.
\]
A solution to the problem
\[
w(z)=(Tw)(z)+g(z)
\]
also solves the equation $\overline{\partial}w(z)=\mu(z)\partial w(z)+\gamma(z)w(z)+\delta(z)$ with $w(0)=0$ and $\partial w(0)=1$. In lecture 4 of \cite{Bers-lecture} it is shown that $T$ defines a bounded linear operator 
\[
T:C^{1,\alpha}(B_R(0))\longrightarrow C^{1,\alpha}(B_R(0))
\]
with
\[
\|T\|\le\,\mbox{const}\cdot R^{\alpha}=\theta\ ,\ \mbox{and $\theta<1$ for small $R>0$}
\]
so that the series $g+\sum_{k=1}^{\infty}T^kg$ converges and the limit $w$ satisfies $w=Tw+g$. Another useful fact is the following: Assume, $T_1,T_2$ are operators as above with coefficient functions $\mu_1,\gamma_1$ and $\mu_2,\gamma_2$ respectively. Then
\[
\|T_2-T_1\|\le\,c\,(\|\mu_2-\mu_1\|_{C^{0,\alpha}(B_R(0))}+\|\gamma_2-\gamma_1\|_{C^{0,\alpha}(B_R(0))})
\] 
for a suitable constant $c>0$ depending on $\alpha$ and $R$. This is only implicitly proved in \cite{Bers-lecture}, so we sketch the proof of this inequality. We have
\begin{eqnarray*}
 (T_2-T_1)h(z) & = & A((\mu_2-\mu_1)\partial h+(\gamma_2-\gamma_1)h)(z)-\\
& & -z\,\Gamma((\mu_2-\mu_1)\partial h+(\gamma_2-\gamma_1)h)(0),
\end{eqnarray*}
\begin{eqnarray*}
 \partial((T_2-T_1)h)(z) & = & \Gamma((\mu_2-\mu_1)\partial h+(\gamma_2-\gamma_1)h)(z)-\\
& & -\Gamma((\mu_2-\mu_1)\partial h+(\gamma_2-\gamma_1)h)(0)
\end{eqnarray*}
and
\[
 \bar{\partial}((T_2-T_1)h)(z)=(\mu_2-\mu_1)(z)\partial h(z)+(\gamma_2-\gamma_1)(z)h(z).
\]
We will need inequalities (21)--(24) from \cite{Chern}. Adapted to our notation they look as follows with $z,z_1,z_2\in B_R(0)$
\begin{eqnarray*}
 |(Ah)(z)| & \le & 4R\|h\|_{C^0(B_R(0))}\\
|(\Gamma h)(z)| & \le & \frac{2^{\alpha+1}}{\alpha}R^{\alpha}\|h\|_{C^{0,\alpha}(B_R(0))}\\
\frac{|(Ah)(z_2)-(Ah)(z_1)|}{|z_2-z_1|^{\alpha}} & \le & 2\|h\|_{C^0(B_R(0))}+\frac{2^{\alpha+2}}{\alpha}R^{\alpha}\|h\|_{C^{0,\alpha}(B_R(0))}\\
\frac{|(\Gamma h)(z_2)-(\Gamma h)(z_1)|}{|z_2-z_1|^{\alpha}} & \le & C_{\alpha}\,\|h\|_{C^{0,\alpha}(B_R(0))}.
\end{eqnarray*}
Recalling that
\[
 \|h\|_{C^{1,\alpha}(B_R(0))}:=\|h\|_{C^0(B_R(0))}+\|\partial h\|_{C^{0,\alpha}(B_R(0))}+\|\bar{\partial}h\|_{C^{0,\alpha}(B_R(0))}
\]
and
\[
 \|k\|_{C^{0,\alpha}(B_R(0))}:=\|k\|_{C^0(B_R(0))}+\sup_{z_1\neq z_2}\frac{|k(z_2)-k(z_1)|}{|z_2-z_1|^{\alpha}}
\]
and that the H\"{o}lder norm satisfies
\[
 \|hk\|_{C^{0,\alpha}(B_R(0))}\le C\ \|h\|_{C^{0,\alpha}(B_R(0))}\,\|k\|_{C^{0,\alpha}(B_R(0))}
\]
for a suitable constant $C$ depending only on $\alpha$ and $R$ the asserted inequality for the operator norm of $T_2-T_1$ follows. In the same way we obtain
\[
\|g_2-g_1\|_{C^{1,\alpha}(B_R(0))}\le\,c\,\|\delta_2-\delta_1\|_{C^{0,\alpha}(B_R(0))}.
\]
Since 
\begin{eqnarray*}
\|w_2-w_1\|_{C^{1,\alpha}(B_R(0))} & \le & \|(T_2-T_1)w_2\|_{C^{1,\alpha}(B_R(0))}+\\
 & & +\,\theta\,\|w_2-w_1\|_{C^{1,\alpha}(B_R(0))}+\|g_2-g_1\|_{C^{1,\alpha}(B_R(0))}
\end{eqnarray*}
and $\theta<1$ we obtain the assertion of the theorem for $k=1$. Because derivatives of $w$ satisfy again an equation of the form $\overline{\partial}w(z)=\mu(z)\partial w(z)+\gamma(z)w(z)+\delta(z)$, we can proceed by iteration. This is carried out in lecture 5 of \cite{Bers-lecture}.
\qed\\

\subsection{A uniform $L^2$--bound for the harmonic forms and uniform convergence}

\begin{proposition}\label{L^2-bound-for-gamma}
Let $(S,j_0,\Gamma,\tilde{u}_0,\gamma_0)$ be a solution of the PDE (\ref{eq1}) defined on $\dot{S}$ which is everywhere transverse to the Reeb vector field. Assume that $(S,j_f,\Gamma,\tilde{u}=(a,u),\gamma)$ is another smooth solution where $u$ is given by
\[
u(z)=\phi_{f(z)}(u_0(z))
\]
for a suitable smooth bounded function $f:S\rightarrow {\mathbb R}$. Then we have
\begin{equation}\label{L^2-bound-for-gamma-2}
\|\gamma\|_{L^2,j_f}\le \|u_0^{\ast}\lambda\|_{L^2,j_f}
\end{equation}
where 
\[
\|\sigma\|_{L^2,j_f}:=\left(\int_{\dot{S}}\sigma\circ j_f\wedge\sigma\right)^{\frac{1}{2}}
\]
(with $\sigma$ a 1--form on $\dot{S}$).
\end{proposition}

{\bf Proof:}\\
Using the differential equation $u^{\ast}\lambda\circ j_f=da+\gamma$ and $u^{\ast}\lambda=u_0^{\ast}\lambda+df$, we compute
\begin{eqnarray*}
\int_{\dot{S}}u^{\ast}\lambda\wedge \gamma & = & \int_{\dot{S}}u_0^{\ast}\lambda\wedge \gamma+\int_{\dot{S}}d(f\gamma)\\
 & = & \int_{\dot{S}}u_0^{\ast}\lambda\wedge \gamma
\end{eqnarray*}
and
\begin{eqnarray*}
\int_{\dot{S}}u^{\ast}\lambda\wedge\gamma & = & \int_{\dot{S}}da\wedge\gamma\circ j_f-\|\gamma\|^2_{L^2,j_f}\\
& = & -\|\gamma\|^2_{L^2,j_f}.
\end{eqnarray*}
The integral $\int_Sd(f\gamma)$ vanishes by Stokes' theorem since $f\gamma$ is a smooth 1-form on the closed surface $S$. The form $da\wedge \gamma\circ j_f$ is not smooth on $S$, but the integral vanishes anyway for the following reason. As we have proved in appendix \ref{local-computations} the form $\gamma\circ j_f$ is bounded near the punctures, hence in local coordinates near a puncture it is of the form
\[
\sigma=F(w_1,w_2)dw_1+F_2(w_1,w_2)dw_2\ ,\ w_1+iw_2\in{\C}
\]
where $F_1,F_2$ are smooth except possibly at the origin but bounded. Passing to polar coordinates via
\[
\phi:[0,\infty)\times S^1\longrightarrow {\C}\backslash\{0\}
\]
\[
\phi(s,t)=e^{-(s+it)}=w_1+iw_2
\]
we see that $\phi^{\ast}\sigma$ has to decay at the rate $e^{-s}$ for large $s$. The form $da$ has $\gamma_1(r(s))\,ds$ as its leading term. Computing the integral $\int_{\Gamma}a(\gamma\circ j_f)$ over small loops $\Gamma$ around the punctures and using Stokes' theorem we conclude that the contribution from neighborhoods of the punctures can be made arbitrarily small. Therefore, the integral $\int_{\dot{S}}da\wedge \gamma\circ j_f$ must vanish.\\
If $\Omega$ is a volume form on $S$ then we may write $u_0^{\ast}\lambda\wedge\gamma=g\cdot\Omega$ for a suitable smooth function $g$. Defining 
\[
\int_{\dot{S}}|u_0^{\ast}\lambda\wedge\gamma|:=\int_{\dot{S}}|g|\,\Omega
\]
we have
\begin{eqnarray*}
\|\gamma\|^2_{L^2,j_f} & = & \left|\int_{\dot{S}}u_0^{\ast}\lambda\wedge\gamma\right|\\
 & \le & \int_{\dot{S}}|u_0^{\ast}\lambda\wedge\gamma|\\
 & \le & \|u_0^{\ast}\lambda\|_{L^2,j_f}\|\gamma\|_{L^2,j_f}
\end{eqnarray*}
which implies the assertion.\qed

We resume the proof of the compactness result, theorem \ref{compactness-result}. All the considerations which follow are local. The task is to improve the regularity of the limit $f_{\tau_0}$ and the nature of the convergence $f_{\tau}\rightarrow f_{\tau_0}$. Because the proof is somewhat lengthy we organize it in several steps. For $\tau<\tau_0$ let now 
\[
\alpha_{\tau}:B\longrightarrow U_{\tau}\subset {\C}
\]
be the conformal transformations as in the previous section, i.e.
\[
T\alpha_{\tau}(z)\circ i=j_{\tau}(\alpha_{\tau}(z))\circ T\alpha_{\tau}(z)\ ,\ z\in B.
\]
The $L^{\infty}$-bound (\ref{Linfty-bound}) on the family of functions $(f_{\tau})$ and the above $L^2$--bound imply convergence of the harmonic forms $\alpha^{\ast}_{\tau}\gamma_{\tau}$ after maybe passing to a subsequence:
\begin{proposition}\label{convergence-of-gammas}
Let $\tau'_k$ be a sequence converging to ${\tau_0}$ and $B'=B_{\varepsilon'}(0)$ with $\overline{B'}\subset B$. Then there is a subsequence $(\tau_k)\subset(\tau'_k)$ such that the harmonic 1--forms $\alpha_{\tau_k}^{\ast}\gamma_{\tau_k}$ converge in $C^{\infty}(B')$.
\end{proposition}
{\bf Proof:}\\
First, the harmonic 1--forms $\alpha_{\tau}^{\ast}\gamma_{\tau}$ satisfy the same $L^2$--bound as in proposition \ref{L^2-bound-for-gamma}:
\begin{eqnarray*}
\|\alpha_{\tau}^{\ast}\gamma_{\tau}\|^2_{L^2(B)} & = & \int_{B}\alpha_{\tau}^{\ast}\gamma_{\tau}\circ i\wedge\alpha^{\ast}_{\tau}\gamma_{\tau}\\
& = & \int_{B}\alpha_{\tau}^{\ast}(\gamma_{\tau}\circ j_{\tau})\wedge\alpha_{\tau}^{\ast}\gamma_{\tau}\\
& = & \int_{U_{\tau}}\gamma_{\tau}\circ j_{\tau}\wedge\gamma_{\tau}\\
& \le & \|u_0^{\ast}\lambda\|_{L^2,j_{\tau}}\\
& \le & C,
\end{eqnarray*}
where $C$ is a constant only depending on $\lambda$ and $u_0$ since $$\sup_{\tau}\|j_{\tau}\|_{L^{\infty}(\dot{S})}<\infty.$$
We write 
\[
\alpha^{\ast}_{\tau}\gamma_{\tau}=h_{\tau}^1\,ds+h_{\tau}^2\,dt
\]
where $h_{\tau}^k$, $k=1,2$ are harmonic and bounded in $L^2(B)$ independent of $\tau$. If $y\in B$ and $B_R(y)\subset\overline{B_R(y)}\subset B$ then the classical mean value theorem
\[
h^k_{\tau}(y)=\frac{1}{\pi R^2}\int_{B_R(y)}h_{\tau}^k(x)dx
\]
implies that for any ball $B_{\delta}=B_{\delta}(y)$ with $B_{\delta}\subset\overline{B_{\delta}}\subset B$ we have the rather generous estimate
\[
\|h^k_{\tau}\|_{C^0({B_{\delta}(y)})}\le\,\frac{1}{\sqrt{\pi}\,\delta}\,\|h^k_{\tau}\|_{L^2(B)}\le \,\frac{\sqrt{C}}{\sqrt{\pi}\,\delta}.
\]
With $y\in B$ and $\nu=(\nu_1,\nu_2)$ being the unit outer normal to $\partial B_{\delta}(y)$ we get
\begin{eqnarray*}
\pas h_{\tau}^k(y) & = & \frac{1}{\pi {\delta}^2}\int_{B_{\delta}(y)}\pas h_{\tau}^k(x)dx\\
& = & \frac{1}{\pi {\delta}^2}\int_{B_{\delta}(y)}\mbox{div}(h^k_{\tau},0)\,dx\\
& = & \frac{1}{\pi {\delta}^2}\int_{\partial B_{\delta}(y)}h^k_{\tau}\,\nu_1\,ds\\
\end{eqnarray*}
and
\begin{eqnarray*}
|\nabla h^k_{\tau}(y)| & = & \frac{1}{\pi {\delta}^2}\left|\int_{\partial B_{\delta}(y)}h^k_{\tau}\,\nu\,ds\right| \\
& \le & \frac{2}{\delta}\|h_{\tau}^k\|_{C^0(\overline{B_{\delta}(y)})}
\end{eqnarray*}
so that for $B'=B_{\varepsilon'}$, $r$ being the radius of $B$ and $\delta=r-\varepsilon'$ 
\[
\|\nabla h^k_{\tau}\|_{C^0(\overline{B'})} \le\frac{2\sqrt{C}}{\sqrt{\pi}\,\delta^2}.
\]
By iterating this procedure on nested balls we obtain $\tau$--uniform $C^0(\overline{B'})$ bounds on all derivatives. Convergence as stated in the proposition then follows from the Ascoli--Arzela theorem.
\qed

\subsection{A uniform $L^p$--bound for the gradient}

The first step of the regularity story is showing that the gradients of $a_{\tau}+if_{\tau}$ are uniformly bounded in $L^p(B')$ for some $p>2$ and for any ball $B'$ with $\overline{B'}\subset B$. It will become apparent later why this gradient bound is necessary. Since we do not have a lot to start with, the proof will be indirect.
Recall the differential equation (\ref{IFT-eq1}):
\[
\bar{\partial}_{j_{\tau}}(a_{\tau}+if_{\tau})=u_0^{\ast}\lambda\circ j_{\tau}-i(u_0^{\ast}\lambda)-\gamma_{\tau}-i(\gamma_{\tau}\circ j_{\tau}),
\]
where
\begin{eqnarray*}
 j_{\tau}(z)& = & (\pil Tu_0(z))^{-1}(T\phi_{f_{\tau}(z)}(u_0(z)))^{-1}\cdot  \\
 & & \cdot J(\phi_{f_{\tau}(z)}(u_0(z)))\,T\phi_{f_{\tau}(z)}(u_0(z))\pil Tu_0(z).
\end{eqnarray*}
We set
\[
\phi_{\tau}(z):=a_{\tau}(z)+i\,f_{\tau}(z)\ ,\ z\in U_{\tau}.
\]
so that for $z\in B$:
\begin{eqnarray}\label{transformed-PDE}
\overline{\partial}(\phi_{\tau}\circ\alpha_{\tau})(z) & = & \overline{\partial}_{j_{\tau}}\phi_{\tau}(\alpha_{\tau}(z))\circ\partial_s\alpha_{\tau}(z)\nonumber \\
 & = & (u_0^{\ast}\lambda\circ j_{\tau}-i(u_0^{\ast}\lambda))_{\alpha_{\tau}(z)}\circ\partial_s\alpha_{\tau}(z)-\nonumber\\
 & & -\left((\alpha_{\tau}^{\ast}\gamma_{\tau})(z)\cdot\frac{\partial}{\partial s}+i(\alpha_{\tau}^{\ast}\gamma_{\tau})(z)\cdot\frac{\partial}{\partial t}\right) \\
 &=:& \hat{F}_{\tau}(z)+\hat{G}_{\tau}(z)\nonumber\\
 &=:&\hat{H}_{\tau}(z),\nonumber
\end{eqnarray}
and 
\begin{equation}\label{L^p-bound-on-Ftau}
\sup_{\tau}\|\hat{F}_{\tau}\|_{L^p(B)}<\infty\ \mbox{for some $p>2$}
\end{equation}
since $\alpha_{\tau}\rightarrow\alpha_{\tau_0}$ in $W^{1,p}(B)$ and $\sup_{\tau}\|j_{\tau}\|_{L^{\infty}}<\infty$. We also have
\begin{equation}\label{L^p-bound-on-Gtau}
\sup_{\tau}\|\hat{G}_{\tau}\|_{C^k(B')}<\infty
\end{equation}
for any ball $B'\subset\overline{B'}\subset B$ and any integer $k\ge 0$ in view of proposition \ref{convergence-of-gammas} (the proposition asserts uniform convergence after passing to a suitable subsequence, but uniform bounds on all derivatives are established in the proof). We claim now that for every ball $B'\subset\overline{B'}\subset B$ there is a constant $C_{B'}>0$ such that
\begin{equation}\label{gradient-bound}
\|\nabla(\phi_{\tau}\circ \alpha_{\tau})\|_{L^p(B')}\,\le\,C_{B'}\ \forall\ \tau\in[0,\tau_0).
\end{equation}
Arguing indirectly, we may assume that there is a sequence $\tau_k\nearrow {\tau_0}$ such that
\begin{equation}\label{stuff-blows-up}
\|\nabla(\phi_{\tau_k}\circ \alpha_{\tau_k})\|_{L^p(B')}\rightarrow \infty\ \mbox{for some ball}\ B'\subset\overline{B'}\subset B.
\end{equation}
Define now
\[
\ve_k:=\inf\{\ve>0\,|\,\exists\, x\in B'\,:\,\|\nabla(\phi_{\tau_k}\circ\alpha_{\tau_k})\|_{L^p(B_{\ve}(x))}\ge\ve^{\frac{2}{p}-1}\}
\]
which are positive numbers since $\varepsilon^{\frac{2}{p}-1}\rightarrow+\infty$. Because we assumed (\ref{stuff-blows-up}) we must have $\inf_k\ve_k=0$, hence we will assume that $\ve_k\rightarrow 0$. Otherwise, if $\ve_0=\frac{1}{2}\inf_k\ve_k>0$ then we cover $\overline{B'}$ with finitely many balls of radius $\ve_0$, and we would get a $k$--uniform $L^p$--bound on each of them contradicting (\ref{stuff-blows-up}). We claim that
\begin{equation}\label{first-est}
\|\nabla(\phi_{\tau_k}\circ\alpha_{\tau_k})\|_{L^p(B_{\ve_k}(x))}\le\ve_k^{\frac{2}{p}-1}\ \forall\ x\in B'.
\end{equation}
Otherwise we could find $y\in B'$ so that
\[
\|\nabla(\phi_{\tau_k}\circ\alpha_{\tau_k})\|_{L^p(B_{\ve_k}(y))}>\ve_k^{\frac{2}{p}-1},
\]
and we would still have the same inequality for a slightly smaller $\ve_k'<\ve_k$ contradicting the definition of $\ve_k$. We now claim that there is a point $x_k\in B'$ with
\begin{equation}\label{second-est}
\|\nabla(\phi_{\tau_k}\circ\alpha_{\tau_k})\|_{L^p(B_{\ve_k}(x_k))}=\ve_k^{\frac{2}{p}-1}.
\end{equation}
Indeed, pick sequences $\delta_l\searrow\ve_k$ and $y_l\in B'$ so that
\[
\|\nabla(\phi_{\tau_k}\circ\alpha_{\tau_k})\|_{L^p(B_{\delta_l}(y_l))}\ge \delta_l^{\frac{2}{p}-1}.
\]
We may assume that the sequence $(y_l)$ converges. Denoting its limit by $x_k$, we obtain
\[
\|\nabla(\phi_{\tau_k}\circ\alpha_{\tau_k})\|_{L^p(B_{\ve_k}(x_k))}\ge\ve_k^{\frac{2}{p}-1}
\]
and (\ref{second-est}) follows from (\ref{first-est}). Hence there is a sequence $(x_k)\subset B'$ for which (\ref{second-est}) holds. We may assume that the sequence $(x_k)\subset B'$ converges and without loss of generality also that $\lim_{k\rightarrow\infty}x_k=0$. Let $R>0$, and we define for $z\in B_R(0)$ the functions 
\[
\xi_k(z):=(\phi_{\tau_k}\circ\alpha_{\tau_k})(x_k+\ve_k(z-x_k))
\]
which makes sense if $k$ is sufficiently large. The transformation 
\[
\Phi:x\longmapsto x_k+\ve_k(x-x_k)
\]
satisfies $\Phi(B_1(x_k))=B_{\ve_k}(x_k)$ and $\Phi(B_1(y))\subset B_{\ve_k}(x_k+\ve_k(y-x_k))$ so that
\begin{eqnarray*}
 \int_{B_{\varepsilon_k}(x_k)}|\nabla(\phi_{\tau_k}\circ \alpha_{\tau_k})(x)|^pdx & = & \varepsilon_k^2\int_{B_1(x_k)}|\nabla(\phi_{\tau_k}\circ\alpha_{\tau_k})(x_k+\varepsilon_k(z-x_k))|^pdz\\
& = & \varepsilon_k^2\int_{B_1(x_k)}\varepsilon_k^{-p}|\nabla\xi_k(z)|^pdz
\end{eqnarray*}
and
\begin{eqnarray}\label{third-est}
\|\nabla \xi_k\|_{L^p(B_{1}(x_k))} & = &  \ve_k^{1-\frac{2}{p}}\|\nabla(\phi_{\tau_k}\circ\alpha_{\tau_k})\|_{L^p(B_{\ve_k}(x_k))}\\
 &=& 1\ \mbox{by equation}\ (\ref{second-est})\nonumber
\end{eqnarray}
and for any $y$ for which $\xi_k|_{B_1(y)}$ is defined and large enough $k$
\begin{equation}\label{another-est}
\|\nabla\xi_k\|_{L^p(B_1(y))}\le\ve_k^{1-\frac{2}{p}}\|\nabla(\phi_{\tau_k}\circ\alpha_{\tau_k})\|_{L^p(B_{\ve_k}(x_k+\ve_k(y-x_k)))}\le 1
\end{equation}
by (\ref{first-est}). The functions $\xi_k$ satisfy the equation
\begin{equation}\label{PDE-for-xi}
\bar{\partial}\xi_k(z)=\ve_k\hat{F}_{\tau_k}(x_k+\ve_k(z-x_k))+\ve_k\hat{G}_{\tau_k}(x_k+\ve_k(z-x_k))=:H_{\tau_k}(z)
\end{equation}
and for every $R>0$
\begin{equation}\label{grad-inequalities}
\sup_k\|\nabla\xi_k\|_{L^p(B_R(0))}<\infty\ ,\ \|\nabla\xi_k\|_{L^p(B_2(0))}\ge 1
\end{equation}
because of (\ref{another-est}) and (\ref{third-est}) since $B_1(x_k)\subset B_2(0)$ for large $k$. The upper bound on $\|\nabla\xi_k\|_{L^p(B_R(0))}$ depends on how many balls $B_1(y)$ are needed to cover $B_R(0)$. We compute for $\rho>0$
\[
\|H_{\tau_k}\|_{L^p(B_\rho(x_k))}=\ve_k^{1-\frac{2}{p}}\|\hat{H}_{\tau_k}\|_{L^p(B_{\rho\ve_k}(x_k))},
\]
with $\hat{H}_{\tau_k}$ as in (\ref{transformed-PDE}). We conclude from, $p>2$, (\ref{L^p-bound-on-Ftau}) and (\ref{L^p-bound-on-Gtau}) that 
\[
 \|H_{\tau_k}\|_{L^p(B_R(0))}\longrightarrow 0
\]
for any $R>0$ as $k\rightarrow\infty$. Defining
\[
X^{l,p}:=\{\psi\in W^{l,p}(B,\C^2)\,|\,\psi(0)=0\,,\,\psi(\partial B)\subset\R^2\}\ ,\ l\ge 1,\ B\subset{\C}\ \mbox{a ball}
\]
the Cauchy-Riemann operator
\[
\bar{\partial}:X^{l,p}\longrightarrow W^{l-1,p}(B,\C^2)
\]
is a bounded bijective linear map. By the open mapping principle we  have the following estimate:
\begin{equation}\label{regularity-for-dell-bar}
\|\psi\|_{l,p,B}\le\,C\,\|\overline{\partial}\psi\|_{l-1,p,B}\ \forall\ \psi\in X^{l,p}.
\end{equation}
Let $R'\in(0,R)$. Pick now a smooth function $\beta:{\R}^2\rightarrow[0,1]$ with $\beta|_{B_{R'}(0)}\equiv 1$ and $\mbox{supp}(\beta)\subset B_R(0)$. Define
\[
\zeta_k(z):=\,\mbox{Re}(\xi_k(z)-\xi_k(0))+i\beta(z)\,\mbox{Im}(\xi_k(z)-\xi_k(0)).
\]
We note that 
\[
\sup_k\|\,\mbox{Im}(\xi_k)\|_{L^p(B_R(0))}\le C_R
\]
with a suitable constant $C_R>0$ because of the uniform bound
\[
\sup_{k,R}\|\,\mbox{Im}(\xi_k)\|_{L^{\infty}(B_R(0))}<\infty.
\]
Using (\ref{regularity-for-dell-bar}) we then obtain
\begin{eqnarray}\label{regularity-for-xi}
\|\xi_k-\xi_k(0)\|_{1,p,B_{R'}(0)} & \le & \|\zeta_k\|_{1,p,B_R(0)}\nonumber\\
 & \le & C\,\|\overline{\partial}\zeta_k\|_{L^p(B_R(0))}\nonumber\\
 & \le & C_{R'}\Big(\|H_{\tau_k}\|_{L^p(B_R(0))}+\|\nabla\xi_k\|_{L^p(B_R(0))}+\\
 & & +\|\,\mbox{Im}(\xi_k)-\,\mbox{Im}(\xi_k)(0)\|_{L^p(B_R(0))}\Big)\nonumber
\end{eqnarray}
because of
\[
\overline{\partial}\zeta_k=H_{\tau_k}+i(\beta-1)\overline{\partial}(\,\mbox{Im}(\xi_k))+i\,\overline{\partial}\beta\,\Big(\mbox{Im}(\xi_k)-\mbox{Im}(\xi_k)(0)\Big).
\]
Hence the sequence $(\xi_k-\xi_k(0))$ is uniformly bounded in $W^{1,p}(B_{R'}(0))$, in particular it has a subsequence which converges in $C^{\alpha}(\overline{B_{R'}(0)})$ for $0<\alpha< 1-\frac{2}{p}$ and also in $L^p(B_{R'}(0))$. For $R''\in(0,R')$ we use now the regularity estimate
\begin{eqnarray}\label{show-1-p-convergence}
\|\xi_l-\xi_k-(\xi_l-\xi_k)(0)\|_{1,p,B_{R''}(0)}&\le &\,c\,\|H_{\tau_l}-H_{\tau_k}\|_{L^p(B_{R'}(0))}+\nonumber \\
 & & +c\,\|\xi_l-\xi_k-(\xi_l-\xi_k)(0)\|_{L^p(B_{R'}(0))},
\end{eqnarray}
where $c=c(p,R',R'')>0$. This follows from (\ref{regularity-for-dell-bar}) applied to $\psi=\beta(\xi_l-\xi_k-(\xi_l-\xi_k)(0))$ where $\beta$ is a smooth cut-off function with support in $B_{R'}$ and $\beta\equiv 1$ on $B_{R''}$. We may then assume that the right hand side of (\ref{show-1-p-convergence}) converges to zero as $k,l\rightarrow\infty$. This argument can be carried out for any triple $0<R''<R'<R$. Hence the sequence $\xi_k-\xi_k(0)$ converges in $W^{1,p}_{loc}({\mathbb C})$ to some limit $\xi:{\C}\rightarrow {\C}$ which solves $\bar{\partial}\xi=0$ in the sense of distributions. Therefore it is an entire holomorphic function. Because the imaginary parts of $\xi_k$ are uniformly bounded this also applies to $\,\mbox{Im}(\xi)$. Liouville's theorem for harmonic functions then implies that $\,\mbox{Im}(\xi)$ must be constant, hence $\xi$ is constant as well. On the other hand, $\xi$ can not be constant since it satisfies $\|\nabla\xi\|_{L^p(B_2(0))}\ge 1$. This contradiction finally disproves our assertion (\ref{stuff-blows-up}). We summarize:

\begin{proposition}\label{lp-bound-for-gradient}
For every ball $B'$ with $\overline{B'}\subset B$ we have
\[
\sup_{\tau}\|\nabla(\phi_{\tau}\circ\alpha_{\tau})\|_{L^p(B')}<\infty.
\]
\end{proposition}\qed\\

\begin{remark}
After establishing the estimates (\ref{grad-inequalities}) for $\nabla\xi_k$ we could have derived a $k$--uniform $W^{1,p}$--bound for $\xi_k$ minus its average $\overline{\xi_k}$ over the ball $B_R(0)$ via Poincar\'e's inequality. We could have derived $W^{1,p}(B_R(0))$ convergence of $\xi_k-\overline{\xi_k}$, but not convergence in $W^{1,p}_{loc}(\C)$ since the sequence $(\xi_k-\overline{\xi_k})$ depends on the choice of the ball $B_R(0)$. Our sequence $\xi_k-\xi_k(0)$ has a convergent subsequence on any ball.
\end{remark}

\subsection{Convergence in $W^{1,p}(B')$}

Pick a sequence $\tau_k\nearrow \tau_0$. We claim that the sequence $(\hat{F}_{\tau_k})$ converges in $L^p(B)$ maybe after passing to a suitable subsequence (recall that so far we only have the uniform bound (\ref{L^p-bound-on-Ftau})). The functions $\hat{F}_{\tau_k}$ converge pointwise almost everywhere after passing to some subsequence: Indeed, the sequence $\{(u_0^{\ast}\lambda\circ j_{\tau_k}-i(u_0^{\ast}\lambda))_{\alpha_{\tau_k}(z)}\}$ converges already pointwise since $j_{\tau_k}$ and $\alpha_{\tau_k}$ do (recall that the sequence $(\alpha_{\tau_k})$ converges in $W^{1,p}(B)$ and therefore uniformly). The sequence $(\pas\alpha_{\tau_k})$ converges in $L^p(B)$ and therefore pointwise almost everywhere after passing to a suitable subsequence. Then, by Egorov's theorem, for any $\delta>0$ there is a subset $E_{\delta}\subset B$ with $|B\backslash E_{\delta}|\le\delta$ so that the sequence $\hat{F}_{\tau_k}$ converges uniformly on $E_{\delta}$. Let $\alpha$ be the $L^p$-limit of the sequence $(\pas\alpha_{\tau_k})$, and let $\varepsilon>0$. We introduce 
\[
C:=2\sup_{0\le \tau<\tau_0}\|(u_0^{\ast}\lambda\circ j_{\tau}-i(u_0^{\ast}\lambda))_{\alpha_{\tau}(z)}\|_{L^{\infty}(B)}. 
\]
Pick now $\delta>0$ sufficiently small such that 
\[
 \|\alpha\|_{L^p(B\backslash E_{\delta})}\,\le\,\frac{\varepsilon}{3\,C}.
\]
Choose now $k_0\ge 0$ so large that for all $k\ge k_0$
\[
 \|\pas\alpha_{\tau_k}-\alpha\|_{L^p(B)}\le\frac{\varepsilon}{3\,C}\ \mbox{and}\ \|\hat{F}_{\tau_k}-\hat{F}_{\tau_l}\|_{L^{\infty}(E_{\delta})}\le\frac{\varepsilon}{3\,|B|}.
\]
Then, if $k,l\ge k_0$,
\begin{eqnarray*}
\|\hat{F}_{\tau_k}-\hat{F}_{\tau_l}\|_{L^p(B)} & \le & \|\hat{F}_{\tau_k}-\hat{F}_{\tau_l}\|_{L^p(E_{\delta})}+\|\hat{F}_{\tau_k}-\hat{F}_{\tau_l}\|_{L^p(B\backslash E_{\delta})}\\
& \le & |E_{\delta}|\,\|\hat{F}_{\tau_k}-\hat{F}_{\tau_l}\|_{L^{\infty}(E_{\delta})}+2\sup_{k\ge k_0}\|\hat{F}_{\tau_k}\|_{L^p(B\backslash E_{\delta})}\\
&\le &|B|\,\|\hat{F}_{\tau_k}-\hat{F}_{\tau_l}\|^p_{L^{\infty}(E_{\delta})}+C\cdot\sup_{k\ge k_0}\|\partial_s\alpha_{\tau_k}\|_{L^p(B\backslash E_{\delta})}\\
& \le & \varepsilon
\end{eqnarray*}
proving the claim.\\ 

Recalling that $\phi_{\tau}=a_{\tau}+if_{\tau}$ and that the family $f_{\tau}$ satisfies a uniform $L^{\infty}$--bound we have
\[
\sup_{\tau}\|\,\mbox{Im}(\phi_{\tau}\circ\alpha_{\tau})\|_{L^{\infty}(B)}<\infty.
\]
Pick now three balls $B'''\subset B''\subset B'\subset B$ such that the closure of one is contained in the next. Our aim is to establish $W^{1,p}(B''')$--convergence of a subsequence of the sequence $(\phi_{\tau_k}\circ\alpha_{\tau_k})$. By proposition \ref{lp-bound-for-gradient} we have a uniform $L^p(B')$--bound on the gradient. If $\beta:{\R}^2\rightarrow[0,1]$ is a smooth function with $\,\mbox{supp}\,(\beta)\subset B'$ and $\beta|_{B''}\equiv 1$ and if 
\[
\zeta_{\tau}=\,\mbox{Re}(\phi_{\tau}\circ\alpha_{\tau}-\phi_{\tau}(0))+i\beta\,\mbox{Im}(\phi_{\tau}\circ\alpha_{\tau}-\phi_{\tau}(0))
\]
then we proceed in the same way as in (\ref{regularity-for-xi}), and we obtain
\[
\|\varphi_k\|_{1,p,B''}\le\,C\,\Big(\|\hat{H}_{\tau_k}\|_{L^p(B')}+\|\nabla(\phi_{\tau_k}\circ\alpha_{\tau_k})\|_{L^p(B')}+\|\;\mbox{Im}(\varphi_k)\|_{L^p(B')}\Big)
\]
where we wrote
\[
\varphi_k:=\phi_{\tau_k}\circ\alpha_{\tau_k}-(\phi_{\tau_k}\circ\alpha_{\tau_k})(0),
\]
and where $C>0$ is a constant only depending on $p, B'$ and $B''$. The sequence $(\varphi_k)$ is then uniformly bounded in $W^{1,p}(B'')$ by proposition \ref{lp-bound-for-gradient}, and it converges in $L^p(B'')$ after passing to a suitable subsequence. We now use the regularity estimate
\begin{eqnarray}\label{CR-regularity-2}
\|\varphi_l-\varphi_k\|_{1,p,B'''}&\le &C\,\Big(\|\hat{F}_{\tau_l}-\hat{F}_{\tau_k}\|_{L^p(B'')}+\\
& & +\|\hat{G}_{\tau_l}-\hat{G}_{\tau_k}\|_{L^p(B'')}+\|\varphi_l-\varphi_k\|_{L^p(B'')}\Big).\nonumber
\end{eqnarray}
Since the right hand side converges to zero as $k,l\rightarrow\infty$ we obtain
\begin{proposition}\label{W1p-loc-convergence}
For every ball $B'\subset\overline{B'}\subset B$ the sequence $(\phi_{\tau_k}\circ\alpha_{\tau_k}-\phi_{\tau_k}(0))$ has a subsequence which converges in $W^{1,p}(B')$.
\end{proposition}\qed

\subsection{Improving regularity using both the Beltrami and the Cauchy-Riemann equations}

In order to improve the convergence of the conformal transformations $\alpha_{\tau_k}$ we need to improve the convergence of the maps $\mu_{\tau_k}\rightarrow \mu_{\tau_0}$ and the regularity of its limit. It is known that the inverses $\alpha^{-1}_{\tau_k}$ of the conformal transformations $\alpha_{\tau_k}$ also satisfy a Beltrami equation \cite{Bers-paper}
\[
\overline{\partial}\alpha_{\tau}^{-1}=\nu_{\tau}\,\partial\alpha^{-1}_{\tau}
\]
where
\[
\nu_{\tau}(z)=-\frac{\partial\alpha_{\tau}(\alpha^{-1}_{\tau}(z))}{\overline{\partial}\overline{\alpha}_{\tau}(\alpha^{-1}_{\tau}(z))}\mu_{\tau}(\alpha^{-1}_{\tau}(z))
\]
(follows from $0=\overline{\partial}(\alpha^{-1}_{\tau}\circ\alpha_{\tau})=\overline{\partial}\alpha^{-1}_{\tau}(\alpha_{\tau})\overline{\partial}\overline{\alpha}_{\tau}+\partial\alpha^{-1}_{\tau}(\alpha_{\tau})\overline{\partial}\alpha_{\tau}$).
After passing to a suitable subsequence we may assume that $\partial\alpha_{\tau_k}$ and $\overline{\partial}\alpha_{\tau_k}$ converge pointwise almost everywhere since they converge in $L^p(B)$. Hence we may assume that the sequence $(\nu_{\tau_k})$ also converges pointwise almost everywhere. We also have
\[
|\nu_{\tau}(z)|\le|\mu_{\tau}(\alpha^{-1}(z))|,
\]
hence $\nu_{\tau}$ satisfies the same $L^{\infty}$--bound as $\mu_{\tau}$. By lemma \ref{lemma-6.4} we conclude that 
\[
\alpha_{\tau_k}^{-1}\longrightarrow\alpha^{-1}_{1}\ \mbox{in}\ W^{1,p}(B)
\]
with the same $p>2$ as in lemma \ref{lemma-6.4} applied to the functions $\mu_{\tau}$. After passing to some subsequence, the sequence 
\[
(\varphi_k):=(\phi_{\tau_k}-a_{\tau_k}(0))\circ\alpha_{\tau_k}
\]
converges in $W^{1,p}(B')$ for any ball $\overline{B'}\subset B$ by proposition \ref{W1p-loc-convergence}. Indeed, the expression $\phi_{\tau_k}\circ\alpha_{\tau_k}-\phi_{\tau_k}(0)$ and $\varphi_k$ differ by a constant term $if_{\tau_k}(0)$, but the sequence $(if_{\tau_k}(0))$ has a convergent subsequence.\\
We would like to derive a decent notion of convergence for the sequence $(\varphi_{\tau_k}\circ\alpha^{-1}_{\tau_k})$, but the space $W^{1,p}$ is not well--behaved under compositions. The composition of two functions of class $W^{1,p}$ is only in $W^{1,p/2}$. Since we can not choose $p>2$ freely we rather carry out the argument in H\"{o}lder spaces. By the Sobolev embedding theorem and Rellich compactness we may assume that the sequences $(\varphi_k)$ and $(\alpha_{\tau_k}^{-1})$ converge in $C^{0,\alpha}(B')$ for any ball $B'\subset\overline{B'}\subset B$ and $0<\alpha\le 1-\frac{2}{p}$. We conclude from the inequality
\[
\|f\circ g\|_{C^{0,\gamma\delta}(B')}\le\|f\|_{C^{0,\gamma}(B')}\|g\|_{C^{0,\delta}(B')}\,\forall\,f\in C^{0,\gamma}(B')\,,\,g\in C^{0,\delta}(B')\]
where $0<\gamma,\delta\le 1 $ that the sequence $(\phi_{\tau_k}-a_{\tau_k}(0))$ converges in $C^{0,\alpha^2}(B')$. In particular, any sequence $(f_{\tau_k})$, $\tau_k\nearrow \tau_0$, now converges in the $C^{0,\alpha^2}$--norm to $f_{\tau_0}$. H\"{o}lder spaces are well--behaved with respect to multiplication, i.e.
\[
\|fg\|_{C^{0,\gamma}(B')}\le\,2\,\|f\|_{C^{0,\gamma}(B')}\|g\|_{C^{0,\gamma}(B')},
\] 
and composition with a fixed smooth function maps $C^{0,\gamma}(B')$ into itself. It then follows from the definition of the complex structure $j_{\tau}$ and from the definition of $\mu_{\tau}$ that $\mu_{\tau_k}\rightarrow \mu_{\tau_0}$ in the $C^{0,\alpha^2}$--norm as well. We conclude from theorem \ref{classical-Bers}, the classical regularity result for the Beltrami equation, that $\alpha_{\tau_k}\rightarrow\alpha_{\tau_0}$ in the $C^{1,\alpha^2}$--norm. The regularity estimate for the Cauchy-Riemann operator (\ref{CR-regularity-2}) is also valid in H\"{o}lder spaces, i.e.
\begin{eqnarray*}
\|\varphi_l-\varphi_k\|_{C^{k+1,\gamma}(B''')}&\le &C\,\Big(\|\hat{F}_{\tau_l}-\hat{F}_{\tau_k}\|_{C^{k,\gamma}(B'')}+\\
& & +\|\hat{G}_{\tau_l}-\hat{G}_{\tau_k}\|_{C^{k,\gamma}(B'')}+\|\varphi_l-\varphi_k\|_{C^{k,\gamma}(B'')}\Big).\nonumber
\end{eqnarray*}
The sequence $(\hat{F}_{\tau_k})$ now converges in the $C^{0,\alpha^2}$--norm, and the sequence $(\hat{G}_{\tau_k})$ converges in any H\"{o}lder norm. We obtain with the above regularity estimate $C^{1,\alpha^2}$--convergence of the sequence $(\varphi_k)$, and composing with $\alpha_{\tau_k}^{-1}$ yields $C^{1,\alpha^4}$--convergence of $(f_{\tau_k})$ and $(\mu_{\tau_l})$. Invoking theorem \ref{classical-Bers} again then improves the convergence of the transformations $\alpha_{\tau_k}\,,\,\alpha_{\tau_k}^{-1}$ to $C^{2,\alpha^4}$. We now iterate the procedure using the regularity estimate for the Cauchy-Riemann operator in H\"{o}lder space and the estimate for the Beltrami equation in theorem \ref{classical-Bers}.\\ 

Theorem \ref{compactness-result} follows if we apply the Implicit Function theorem to the limit solution $(S,j_{\tau_0},\tilde{u}_{\tau_0}=(a_{\tau_0},u_{\tau_0}),\gamma_{\tau_0})$, hence we obtain the same limit for every sequence $\{\tau_k\}$, and we obtain convergence in $C^{\infty}$. 

\section{Conclusion}
The following remarks tie together the loose ends and prove the main result, theorem \ref{main-thm}. We start with a closed three dimensional manifold with contact form $\lambda'$. Giroux's theorem, theorem \ref{giroux}, then permits us to change the contact form $\lambda'$ to another contact form $\lambda$ such that $\ker\lambda=\ker\lambda'$ and such that there is a supporting open book decomposition with binding $K$ consisting of periodic orbits of the Reeb vector field of $\lambda$. Invoking proposition \ref{localnormalform} we construct a family of 1-forms $(\lambda_{\delta})_{0\le\delta<1}$ which are contact forms except $\lambda_0$, and the above open book supports $\ker\lambda_{\delta}$ as well if $\delta\neq 0$. By the uniqueness part of Giroux's theorem, $(M,\ker\lambda)$ and $(M,\ker\lambda_{\delta})$ are diffeomorphic for $\delta\neq 0$, hence we may assume without loss of generality that $\lambda=\lambda_{\delta}$. Proposition \ref{first-solution} then permits us to turn the Giroux leaves into holomorphic curves for data associated with the confoliation form $\lambda_0$. Picking one Giroux leaf, the implicit function theorem, theorem \ref{IFT}, then allows us to deform it into a solutions to our PDE (\ref{eq1}) for small $\delta\neq 0$. Leaving such a parameter $\delta$ fixed from now, and denoting the corresponding solution by $(\tilde{u}_0,\gamma_0,j_0)$, theorem \ref{IFT} then delivers more solutions $(\tilde{u}_{\tau},\gamma_{\tau},j_{\tau})_{0\le\tau<\tau_0}$. The leaves $u_{\tau}(\dot{S})$ are all global surfaces of section recalling that they are of the form $u_{\tau}=\phi_{f_{\tau}}(u_0)$. Theorem \ref{IFT} also implies that $f_{\tau}<f_{\tau'}$ if $\tau<\tau'$. The compactness result, theorem \ref{compactness-result}, then implies that there is a 'last' solution for $\tau=\tau_0$ as well, and that either $u_{\tau_0}(\dot{S})$ is disjoint from $u_0(\dot{S})$ or agrees with it. In the latter case, the proof of theorem \ref{main-thm} is complete. In the first case we apply theorem \ref{IFT} again to $(\tilde{u}_{\tau_0},\gamma_{\tau_0},j_{\tau_0})$ producing a larger family of solutions. Because $\tau\mapsto f_{\tau}(z)$ is strictly monotone for each $z\in S$ and because the return time for each point on $u_0(\dot{S})$ is finite, the images of $u_{\tau}$ and $u_0$ must agree for some sufficiently large $\tau$ concluding the proof.

\begin{appendix}
\section{Some local computations near the punctures}\label{local-computations}
 In this appendix we will present some local computations needed for the proof of theorem \ref{IFT}. The issue is to show that the 1-forms
\[
 u_0^{\ast}\lambda\circ j_f-da_0\ \ \mbox{and}\ \ u_0^{\ast}\lambda+da_0\circ j_f
\]
are bounded on $\dot{S}$. We obtain in the second case of the theorem
\begin{eqnarray*}
 u_0^{\ast}\lambda\circ j_f-da_0 & = & u_0^{\ast}\lambda\circ(j_f-j_g)+\gamma_0\\
& = & dg\circ(j_f-j_g)+\gamma_0+v_0^{\ast}\lambda\circ(j_f-j_g).
\end{eqnarray*}
The first case can be treated as a special case: Here the objective is to show that the 1-form $v_0^{\ast}\lambda\circ(j_f-i)=v_0^{\ast}\lambda\circ(j_f-j_0)$ is bounded near the punctures. We again drop the subscript $\delta$ in the notation since we are only concerned with a local analysis near the binding, and all the forms $\lambda_{\delta}$ are identical there. We use coordinates $(\theta,r,\phi)$ near the binding. The contact structure is then generated by
\[
 \eta_1=\frac{\partial}{\partial r}=(0,1,0)\ ,\ \eta_2=-\gamma_2\frac{\partial}{\partial\theta}+\gamma_1\frac{\partial}{\partial \phi}=(-\gamma_2,0,\gamma_1).
\]
The projection onto the contact planes along the Reeb vector field is then given by
\[
 \pi_{\lambda}(v_1,v_2,v_3)=\frac{1}{\mu}(v_1\gamma'_1+v_3\gamma'_2)\,\eta_2+v_2\,\eta_1\ \mbox{with}\ \mu=\gamma_1\gamma'_2-\gamma'_1\gamma_2,
\]
and the flow of the Reeb vector field is given by
\[
 \phi_t(\theta,r,\phi)\,=\,(\theta+\alpha(r)t,r,\phi+\beta(r)t)
\]
where
\[
 \alpha(r)\,=\,\frac{\gamma'_2(r)}{\mu(r)}\ \ \mbox{and}\ \ \beta(r)\,=\,-\frac{\gamma_1'(r)}{\mu(r)}.
\]
The linearization of the flow $T\phi_{\tau}(\theta,r,\phi)$ preserves the contact structure. In the basis $\{\eta_1,\eta_2\}$ it is given by
\[
 T\phi_{\tau}(\theta,r,\phi)\,=\,\left(\begin{array}{cc}1 & 0\\ \tau A(r) & 1\end{array}\right)\ \mbox{with}\ A(r)=\frac{1}{\mu^2(r)}(\gamma''_2(r)\gamma_1'(r)-\gamma_1''(r)\gamma_2'(r)).
\]
The complex structure(s) we chose earlier in (\ref{cplx-str-normal-form}) had the following form near the binding with respect to the basis $\{\eta_1,\eta_2\}$:
\[
 J(\theta,r,\phi)=\left(\begin{array}{cc} 0 & -r\gamma_1(r)\\ \frac{1}{r\gamma_1(r)}& 0\end{array}\right).
\]
The induced complex structure $j_{\tau}$ on the surface is then given by
\[
 j_{\tau}(z)=[\pi_{\lambda}Tv_0(z)]^{-1}[T\phi_{\tau}(v_0(z))]^{-1}J(\phi_{\tau}(v_0(z)))T\phi_{\tau}(v_0(z))\pi_{\lambda}Tv_0(z).
\]
With $v_0(s,t)=(t,r(s),\alpha)$ we find that
\[
 \pi_{\lambda}Tv_0(s,t)=\left(\begin{array}{cc} r'(s)& 0\\0&   \frac{\gamma'_1(s)}{\mu(r(s))}\end{array}\right)
\]
so that
\begin{eqnarray*}
 j_{\tau} & = & \left(\begin{array}{cc} -\tau A(r)r\gamma_1(r)&-\frac{r\gamma_1(r)\gamma'_1(r)}{r'\mu(r)}\\\frac{r'\mu(r)}{r\gamma_1(r)\gamma'_1(r)}(1+\tau^2A^2(r)r^2\gamma_1^2(r))&\tau A(r)r\gamma_1(r)\end{array}\right)\\
& = & \left(\begin{array}{cc} -\tau A(r)r\gamma_1(r)& -1\\1+\tau^2A^2(r)r^2\gamma_1^2(r)& \tau A(r)r\gamma_1(r)\end{array}\right)\\
& = & j_0+\tau\,A(r)\,\gamma_1(r)\left(\begin{array}{cc} -1 & 0 \\ \tau A(r)\gamma_1(r) & 1\end{array}\right).
\end{eqnarray*}
and
\[
j_{\tau}-j_{\sigma}\,=\, A(r) r\gamma_1(r)(\tau-\sigma)\left(\begin{array}{cc} -1 &0\\(\tau+\sigma) A(r) r\gamma_1(r)& 1 \end{array}\right) 
\]
using the fact that $r(s)$ satisfies the differential equation
\[
 r'(s)=\frac{\gamma'_1(r(s))\gamma_1(r(s))r(s)}{\mu(r(s))}.
\]
With $v_0^{\ast}\lambda=\gamma_1(r)\,dt$ we obtain
\begin{eqnarray*}
 v_0^{\ast}\lambda\circ(j_{\tau}-j_{\sigma})|_{(s,t)} & = & (\tau-\sigma)A(r(s))r(s)\gamma_1^2(r(s))\cdot\\
& & \cdot\big[(\tau+\sigma)A(r(s))r(s)\gamma_1(r(s))ds+dt\big].
\end{eqnarray*}
Converting from coordinates $(s,t)$ on the half-cylinder to cartesian coordinates $x+iy=e^{-(s+it)}$ in the complex plane we get with $\rho=\sqrt{x^2+y^2}$
\[
 ds=-\frac{1}{\rho^2}(x\,dx+y\,dy)\ \ \mbox{and}\ \ dt=-\frac{1}{\rho^2}(x\,dy-y\,dx).
\]
Recall that $r(s)=c(s)e^{\kappa s}$ where $c(s)$ is a smooth function which converges to a constant as $s\rightarrow+\infty$, and we assumed that $\kappa\le-\frac{1}{2}$ and that $\kappa\notin{\mathbb Z}$. Another assumption was that $A(r)=O(r)$. Hence $r(s)$ is close to $\rho^{-\kappa}$ if $s$ is large (and $\rho$ is small), and $A(r(s))=O(\rho^{-\kappa})$. Also recall that $\gamma_1(r(s))=O(1)$. Summarizing, we need the expression
\[
A(r(s))r(s)\rho^{-2}\rho=O(\rho^{-2\kappa-1})
\]
to be bounded which amounts to $\kappa\le-\frac{1}{2}$. The same argument applied to the form $dg\circ (j_{\tau}-j_{\sigma})$ leads to the same conclusion.
\end{appendix}


\begin{thebibliography}{99}
\bibitem{Abbas-2}
C. Abbas, Pseudoholomorphic Strips in Symplectizations II: Fredholm Theory and Transversality, Comm. Pure and Appl. Math., Vol. 57, (2004), pp. 1-58
\bibitem{ACH}
C. Abbas, K. Cieliebak, H. Hofer, The Weinstein Conjecture for planar contact structures in dimension three, Commentarii Mathematici Helvetici, 80, (2005), pp. 771-793
\bibitem{AHL-1}
C. Abbas, H. Hofer, S. Lisi, Renormalization and energy quantization in Reeb dynamics, in preparation
\bibitem{AHL-2}
C. Abbas, H. Hofer, S. Lisi, Some applications of a homological perturbed Cauchy-Riemann equation, in preparation
\bibitem{part3}
C. Abbas, Pseudoholomorphic strips in symplectisations III: Embedding properties and Compactness, Journal of Symplectic Geometry, Vol. 2, No. 2, (2004), 219--260
\bibitem{A-notes}
C. Abbas, Introduction to Compactness Results in Symplectic Field Theory, Lecture Notes in preparation
\bibitem{AH}
C. Abbas and H. Hofer, Holomorphic curves and global questions in
contact geometry, to appear in Birkh\"{a}user
\bibitem{Ahlforsbook}
L. Ahlfors, Quasiconformal Mappings, Van Nostrand, Princeton (1966)
\bibitem{Bers-paper}
L. Ahlfors and L. Bers, Riemann's mapping theorem for variable metrics, Annals of Math., Vol 72, 2, (1960)
\bibitem{Alexander}
J. Alexander, A lemma on systems of knotted curves, Proc. Nat. Acad. Sci. USA, {\bf 9}, (1923), 93-95
\bibitem{Bers-lecture}
L. Bers, Riemann Surfaces, Lectures given at New York University, 1957--58
\bibitem{Bers-Nirenberg}
L. Bers and L. Nirenberg, On a representation theorem for linear elliptic systems with discontinuous coefficients and its applications, Convegno intern. sulle equazioni lineari alle derivate parziali, Trieste, 1954
\bibitem{BEHWZ}
F. Bourgeois, Y. Eliashberg, H. Hofer, K. Wysocki, E. Zehnder, Compactness results in Symplectic Field Theory, Geom. and Topology, 7, (2003), 799--888
\bibitem{Chern}
S.-S. Chern, An elementary proof of the existence of isothermal parameters on a surface, Proc. Amer. Math. Soc., Vol 6, No. 5, (1955), pp. 771--782.
\bibitem{Etnyre}
J. Etnyre, Planar open book decompositions and contact structures, IMRN  2004 (2004), no. 79, 4255-4267
\bibitem{Forster}
O. Forster, Riemann Surfaces, Springer 
\bibitem{Giroux}
E. Giroux, G\'{e}om\'{e}trie de contact: de la dimension trois vers les dimensions sup\'{e}rieures, Proceedings of the International Congress of Mathematicians, Vol. II (Beijing, 2002),  405--414, Higher Ed. Press, Beijing, 2002. 
\bibitem{Hofer-Weinstein-conj}
        Hofer H., Pseudoholomorphic curves in symplectizations with applications to the Weinstein conjecture in dimension three. Invent. Math. 114 (1993), no. 3, 515--563.
\bibitem{HZ-book} H. Hofer, E. Zehnder, Symplectic invariants and Hamiltonian dynamics, Birkh\"{a}user Advanced Texts: Basler Lehrb\"{u}cher, Birkh\"{a}user, Basel, 1994
\bibitem{HWZ-embedding} H. Hofer, K. Wysocki and E. Zehnder,
  Properties of pseudo-holomorphic curves in symplectisations II:
  Embedding controls and algebraic invariants, Geom. Funct. Anal. 5
  (1995),  no. 2, 270--328. 
\bibitem{HWZ-asymptotics} H. Hofer, K. Wysocki and E. Zehnder,
  Properties of pseudoholomorphic curves in
  symplectisations I: Asymptotics, Ann. Inst. H. Poincar\'{e} Anal. Non
  Lin\'{e}aire  13  (1996),  no. 3, 337--379.  
\bibitem{HWZ-2}
H. Hofer, K. Wysocki, and E. Zehnder, Properties of pseudoholomorphic curves in symplectizations. II. Geometric and Functional Analysis, volume 5; pp. 270 - 328, (1995)
\bibitem{HWZ-3}
H. Hofer, K. Wysocki, and E. Zehnder, Properties of pseudoholomorphic curves in symplectizations. III. Fredholm theory, Topics in nonlinear analysis, 381-475, Progr. Nonlinear Differential Equations Appl., 35, Birkhauser, Basel, 1999
\bibitem{H-3D}
H. Hofer, Holomorphic curves and real three-dimensional dynamics. In GAFA 2000 (Tel Aviv, 1999), Geom. Funct. Anal., Special Volume, Part II, 2000, 674–704
\bibitem{HWZ-fef}
H. Hofer, K. Wysocki, E. Zehnder, Finite energy foliations of tight three-spheres and Hamiltonian dynamics, Ann. of Math. (2) 157 (2003), no. 1, 125–255.
\bibitem{HWZ-sc}
H. Hofer, K. Wysocki and E. Zehnder, The dynamics on strictly convex energy surfaces in ${\mathbb R}^4$. Ann. Math. (2) 148 (1998), 197–289
\bibitem{Hummel}
  C. Hummel, Gromov's Compactness Theorem for Pseudoholomorphic Curves, Progress in Mathematics, Vol. 151, Birkh\"{a}user Verlag, 1997 
\bibitem{Lax}
P. Lax, Functional Analysis, Wiley, 2001
\bibitem{Lichtenstein}
L. Lichtenstein, Zur Theorie der konformen Abbildung. Konforme Abbildung nichtanalytischer singularit\"{a}tenfreier Fl\"{a}chenst\"{u}cke auf ebene Gebiete, Bull. Acad. Sc. Cracovie, A, (1916)
\bibitem{McDS}
D. McDuff and D. Salamon, J-holomorphic Curves and Symplectic Topology, AMS, 2004
\bibitem{MW}
M. J. Micallef and B. White, The structure of branch points in minimal surfaces and in pseudoholomorphic curves, Ann. of Math. (2)  141  (1995),  no. 1, 35--85
\bibitem{Morrey}
C. B. Morrey, On the solutions of quasilinear elliptic partial differential equations, Trans. Amer. Math. Soc., Vol. 43, (1938), 126--166
\bibitem{Rolfsen}
D. Rolfsen, Knots and Links, Publish or Perish Inc., Houston, Texas
\bibitem{tromba}
A. Tromba, Teichm\"{u}ller Theory in Riemannian Geometry, Lectures in Mathematics, ETH Z\"{u}rich, Birkh\"{a}user Verlag, 1992
\bibitem{Siefring}
R. Siefring, Relative Asymptotic Behavior of Pseudoholomorphic Half-Cylinders, Comm. Pure and Appl. Math., Vol. 61, (2008), pp. 1631-1684
\bibitem{Stein}
E. M. Stein, Singular Integrals and differentiability properties of functions, Princeton University Press, 1970 
\bibitem{Taubes-1}
C.H. Taubes, The Seiberg-Witten equations and the Weinstein conjecture,  Geom. Topol.  11  (2007), 2117--2202 
\bibitem{Taubes-2}
C.H. Taubes, The Seiberg-Witten equations and the Weinstein conjecture II: More closed integral curves of the Reeb vector field, preprint (2007), arXiv [pdf]  	math.SG/0702366
\bibitem{TW}
W. Thurston, H. Winkelnkemper, On the existence of contact forms, Proc. AMS, {\bf 52}, (1975), 345-347
\bibitem{Wendl}
C. Wendl, Open book decompositions and stable Hamiltonian structures, preprint (2008), arXiv:0808.3220v1 [math.SG]
\bibitem{Wendl-2}
C. Wendl, Finite Energy Foliations on Overtwisted Contact Manifolds, Geometry \& Topology 12 (2008) 531-616
\bibitem{Wendl-3}
C. Wendl, Strongly Fillable Contact Manifolds and J-holomorphic Foliations, preprint (2008), arXiv:0806.3193v2 [math.SG]
\end{thebibliography}
\end{document}